\documentclass[11pt]{article}
\usepackage{amscd, amsthm, color}
\usepackage{mathrsfs}
\input epsf
\usepackage{amssymb}
\usepackage{amsmath}
\usepackage{amsfonts}
\input amssym.def
\newtheorem{theorem}{Theorem}[section]

\newtheorem{lemma}[theorem]{Lemma}
\newtheorem{prop}[theorem]{Proposition}

\newtheorem{re}[theorem]{Remark}
\newtheorem{no}[theorem]{Notation}
\newtheorem{definition}[theorem]{Definition}
\theoremstyle{definition}

\definecolor{wco}{rgb}{0.5,0.2,0.3}

\numberwithin{equation}{section}

\begin{document}

\def\beg{\begin}
\def\beq{\begin{equation}}
\def\enq{\end{equation}}

\title{Widths of embeddings of 2-microlocal Besov spaces
\thanks{Partially supported by the Natural Science Foundation of China (Grant No.
10671019) and by Anhui Provincial Natural Science Foundation (No.
090416230).
\newline\indent\ \, E-mail addresses: fanggs@bnu.edu.cn (G. Fang),
 shzhang27@163.com (S. Zhang).}}
\author{Shun Zhang $^{a,\, b}$\ \ and\ \ Gensun Fang
$^{a}$
\\
{\small $^{a}$ School of Mathematical Sciences, Beijing Normal
University, Beijing 100875, China}
\\ {\small $^{b}$ School of Computer Science and Technology,
Anhui University,
 Hefei 230039, Anhui, China} } \maketitle

\begin{abstract}
We consider the asymptotic behaviour of the approximation, Gelfand
and Kolmogorov numbers of compact embeddings between 2-microlocal
Besov spaces with weights defined in terms of the distance to a
$d$-set $U\subset \mathbb{R}^n$. The sharp estimates are shown in
most cases, where the quasi-Banach setting is included.
\end{abstract}

\section{Introduction}

In this paper we investigate the embeddings between 2-microlocal
Besov spaces with one special type of weights from the standpoint of
certain approximation quantities. More precisely, we are interested
in asymptotic behaviour of the approximation, Gelfand and Kolmogorov
numbers. This problem has recently been suggested only for entropy
numbers by Leopold and Skrzypczak \cite{LS11}. First, we recall some
definitions.

Let $\varphi$ be a positive function from the Schwartz space
$\mathcal {S}(\mathbb{R}^n)$ of infinitely differentiable and
rapidly decreasing functions with
\begin{equation}\label{varphi0}
\varphi(x)=1\ \ {\rm for}\ |x|\le 1 \ \ {\rm and}\ \ {\rm supp}\,
\varphi\subset\{x:|x|\le2\}.
\end{equation}
We set $\varphi_0=\varphi$ and $
\varphi_j(x)=\varphi(2^{-j}x)-\varphi(2^{-j+1}x)$ for $j\in
\mathbb{N}$ and $x\in\mathbb{R}^n$. This leads to the smooth dyadic
resolution $\{\varphi_j\}_{j\in\mathbb{N}_0}$ of unity, i.e.,
$\sum_{j=0}^\infty\varphi_j(x)=1,\ x\in\mathbb{R}^n$, so
\begin{equation*}
f=\varphi_0(D)f+\sum\limits_{j=1}^\infty\varphi_j(D)f,
 \quad\quad f\in\mathcal {S}^\prime(\mathbb{R}^n),
\end{equation*}
where
\begin{equation*}
\varphi_j(D)f(x):=(2\pi)^{-n}\iint
e^{i\xi(x-y)}\varphi_j(\xi)f(y)dyd\xi.
\end{equation*}
For a bounded subset $U\subset \mathbb{R}^n$, we denote ${\rm
dist}(x,U)=\inf_{y\in U}|x-y|$, and we define for $s^\prime\in
\mathbb{R}$ the 2-microlocal weights by
\begin{equation}\label{w_U}
w_j(x):=(1+2^j\,{\rm dist}(x,U))^{s^\prime},\ \ j\in\mathbb{N}_0,\
x\in\mathbb{R}^n.
\end{equation}
This type of weight sequences is just a typical example for
admissible weight sequences. We refer to \cite{Ke08,Ke10} for
detailed discussions of a large class of admissible weight
sequences. The case of single weights seems more familiar to us,
cf., e.g., \cite{HS11,KLSS06b,ZF10}.

Given\ $0<p,q\le\infty$\ and $s,s^\prime\in \mathbb{R}$, we define
the 2-microlocal spaces $ B_{p,q}^{s,s^\prime}(\mathbb{R}^n, U)$ by
$$
B_{p,q}^{s,s^\prime}(\mathbb{R}^n, U)=\Big\{f\in\mathcal
{S}^\prime(\mathbb{R}^n)\,:\ \|f\, |\,
B_{p,q}^{s,s^\prime}(\mathbb{R}^n, U)\|<\infty\Big\},
$$
where
$$
\|f\, |\, B_{p,q}^{s,s^\prime}(\mathbb{R}^n, U)\|= \Bigg(
\sum\limits_{j=0}^\infty 2^{jsq}
\|w_j\varphi_j(D)f\,|\,L_p(\mathbb{R}^n)\|^q \Bigg)^{1/q}.
$$

There is an analogous definition for
$F_{p,q}^{s,s^\prime}(\mathbb{R}^n, U)$. Moritoh and Yamada
introduced in \cite{MY96} the spaces
$B_{p,q}^{s,s^\prime}(\mathbb{R}^n, U)$ of homogeneous type in case
when $U\subset \mathbb{R}^n$ is open.

2-Microlocal Besov spaces $B^{s,mloc}_{p,q} (\mathbb{R}^n,w)$ with
more general admissible weights were introduced by Kempka
\cite{Ke08,Ke10b}, and generalized the 2-microlocal spaces
$C^{s,s^\prime}_{x_0} (\mathbb{R}^n)$ introduced by Bony \cite{Bo}
and Jaffard \cite{Ja} in two directions. We refer to
\cite{Ke08,Ke09,Ke10,MX97} for systematic discussions of this
concept, its history and further references.

Following Leopold and Skrzypczak \cite{LS11}, we concentrate on  the
embeddings,
\begin{equation}\label{BB}
B_{p_1,q_1}^{s_1,s_1^\prime}(\mathbb{R}^n, U)\hookrightarrow
B_{p_2,q_2}^{s_2,s_2^\prime}(\mathbb{R}^n, U),
\end{equation}
where $U$ is a $d$-set (the precise definition of $d$-sets will be
given in Section \ref{prel}).

Our main intention in this paper is to find the optimal asymptotic
order of the approximation, Gelfand and Kolmogorov numbers of the
embeddings (\ref{BB}). Our approach is essentially a combination of
\cite{LS11} and \cite{Sk05} with its corrigendum \cite{SV09}. In
particular,
 Leopold and Skrzypczak  \cite{LS11} gave a
necessary and sufficient condition on the parameters and weights of
the 2-microlocal Besov spaces which guarantees compactness of the
embeddings (\ref{BB}), and determined the entropy estimates for such
embeddings. Moreover, our main tools are the use of operator ideals,
see \cite{Car81,Pie78,Pie87}, and the basic estimates of related
widths of the Euclidean ball due to Kashin \cite{Ka77}, Gluskin
\cite{Gl83} and Edmunds and Triebel \cite{ET96} with
\cite{FPRU,GG84,LGM96,Vy08}.

The paper is structured as follows. In Section 2, we introduce
approximation, Gelfand and Kolmogorov numbers, and present our main
results.  Section 3 represents the most dominant part of this paper;
here we adopt a wavelet description of the 2-microlocal Besov spaces
$ B_{p,q}^{s,s^\prime}(\mathbb{R}^n, U)$, and prove their width
estimates of embeddings of related sequence spaces. Finally, in
Section 4, these results will be used to derive the desired
estimates for the function space embeddings under consideration.

Throughout the paper (unless additional restrictions are mentioned)
we suppose that
\begin{equation}\label{agree}
s, s_1, s_2, s^\prime, s_1^\prime, s_2^\prime\in\mathbb{R},\ \ 0<p,
p_1, p_2, q, q_1, q_2\leq\infty,\ \ \delta=s_1-s_2-n(\frac
1{p_1}-\frac 1{p_2})>0.
\end{equation}
 For a real number $a$, we define $a_+=\max(a,0)$. And let $\frac
1{p^*}=(\frac 1{p_2}-\frac 1{p_1})_+$.

\begin{no}
By the symbol ` $\hookrightarrow$'  we denote continuous embeddings.

Identity operators will always be denoted by {\rm id}. Sometimes we
do not indicate the spaces where {\rm id} is considered, and
likewise for other operators.

Let $X$ and $Y$ be complex quasi-Banach spaces and denote by
$\mathcal {L}(X, Y)$ the class of all linear continuous operators
$T:\,X \rightarrow\, Y.$ If no ambiguity arises, we write $\|T\|$
instead of the more exact versions $\|T ~|~ \mathcal {L}(X, Y)\|$ or
$\|T:X\rightarrow Y\|$.

The symbol $a_k \preceq b_k$ means that there exists a constant $c
> 0$\ such
that\ $a_k\le c\,b_k$\ for all\ $k\in\mathbb{N}.$\ And $a_k \succeq
b_k$ stands for $b_k \preceq a_k,$\ while $a_k\sim b_k$ denotes\
$a_k\preceq b_k \preceq a_k.$

All unimportant constants will be denoted by $c$ or $C$, sometimes
with additional indices.
\end{no}

\section{Main results}

We recall the definitions of the approximation, Gelfand and
Kolmogorov numbers, see \cite{Pie78, Pin85}.  We use the symbol
$A\subset\subset B$ if $A$ is a closed subspace of a topological
vector space $B$.

\begin{definition}
Let $T \in\mathcal {L}(X, Y).$
\begin{enumerate}
\item[{\rm (i)}]\
The {\rm $k$th approximation number}\, of~ $T$ is defined by
\begin{equation*}
a_k(T, \,X, \,Y)=\inf\{\|T - A\|:~ A\in \mathcal{L}(X,Y) ~\,{\rm
with~ rank} (A) < k\},\quad k\in \mathbb{N},
\end{equation*}
also written by $a_k(T)$ if no confusion is possible. Here ${\rm
rank} (A)$ is the dimension of the range of the operator $A$.
\item[{\rm (ii)}]\
The {\rm $k$th Kolmogorov number}\, of~ $T$ is defined by
\begin{equation*}
d_k(T, X, Y)=\inf\{\|Q_N^YT\|:\,N\subset\subset Y,\,{\rm dim}
(N)<k\},
\end{equation*}
also written by $d_k(T)$ if no confusion is possible. Here, $Q_N^Y$
stands for the natural surjection of\,\,\,$Y$ onto the quotient
space $Y/N$.
\item[{\rm (iii)}]\
The {\rm $k$th Gelfand number}\, of~ $T$ is  defined by
\begin{equation*}
c_k(T, X, Y)=\inf\{\|TJ_M^X\|:\,M\subset\subset X,\,{\rm codim}
(N)<k\},
\end{equation*}
also written by $c_k(T)$ if no confusion is possible. Here, $J_M^X$
stands for the natural injection of\,\,\,$M$ into $X$.
\end{enumerate}
\end{definition}

Note that the $k$-th approximation, Kolmogorov and Gelfand number
are identical to the $(k-1)$-th linear, Kolmogorov and Gelfand width
of $T$, respectively, see Pinkus \cite{Pin85}.

It is well-known that the operator $T$ is compact if and only if
$\lim_k d_k(T)=0$ or equivalently $\lim_k c_k(T)=0$, but if $\lim_k
a_k(T)=0$, see \cite{Pin85}. The opposite implication for $a_k(T)$
is not true in general.

Both concepts, Kolmogorov and Gelfand numbers, are related to each
other. Namely they are dual to each other in the following sense,
cf. \cite{Pie78, Pin85}: If $X$ and $Y$ are Banach spaces, then
\begin{equation}\label{dualc*d}
c_k(T^\ast)=d_k(T)
\end{equation}
for all compact operators $T\in\mathcal{L}(X, Y)$ and
\begin{equation}\label{duald*c}
d_k(T^\ast)=c_k(T)
\end{equation}
for all $T\in\mathcal{L}(X, Y).$

Both, Gelfand and Kolmogorov numbers, are subadditive and
multiplicative s-numbers, as well as approximation numbers. One may
consult Pietsch \cite{Pie87}(Sections 2.4, 2.5), for the proof in
the Banach space case. And the generalization to $p$-Banach spaces
follows obviously. Let $Y$ be a $p$-Banach space,\ $0<p\le 1$. And
let $s_k$ denote any of the three quantities\ $a_k,\, d_k$ or $c_k$.
Then we collect several common properties of them as follows,
\vspace{-0.2cm}
\begin{enumerate}
\item[]{\rm\bf(PS1)}\ (monotonicity)\,
$\|T\|=s_1(T)\ge s_2(T)\ge\cdots\ge 0$ for all $T\in\mathcal{L}(X,
Y)$,\vspace{-0.2cm}

\item[]{\rm\bf(PS2)}\ (subadditivity)\, $s_{m+k-1}^p(S+T)\leq
s_m^p(S)+s_k^p(T)$\, for all $m, k\in\mathbb{N},\,\,S,
T\in\mathcal{L}(X, Y)$,\vspace{-0.2cm}

\item[]{\rm\bf(PS3)}\ (multiplicativity)\, $s_{m+k-1}(ST)\leq
s_m(S)s_k(T)$\, for all $T\in\mathcal{L}(X, Y)$, $S\in\mathcal{L}(Y,
Z)$

\quad\quad and $m, k\in\mathbb{N},$\, cf. \cite{Pie78}(p. 155),
where $Z$ denotes a quasi-Banach space,\vspace{-0.2cm}

\item[]{\rm\bf(PS4)}\ (rank property)\, ${\rm rank}(T)<k$ if and only if
$s_k(T)=0$, where $T\in\mathcal{L}(X, Y)$.\vspace{-0.2cm}
\end{enumerate}

Moreover, there exist the following relationships: \beq\label{acd}
 c_k(T)\le a_k(T),~~~~ d_k(T)\le a_k(T),~~~ k \in \mathbb{N}.
\enq

 Now we
recall the characterization of compactness of the embeddings under
consideration, which was proved in \cite{LS11}.
\begin{prop}\label{compact}
Let $U$ be a $d$-set, $0 \le d \le n$, $w_{i,j}(x)=(1+2^j\,{\rm
dist}(x,U))^{s_i^\prime},\,i=1,2$,  and
$s^\prime=s_1^\prime-s_2^\prime>0$. Then the embedding (\ref{BB}) is
compact if and only if\ $\delta> d/{p^\ast}$\ and\ $s^\prime>
n/{p^\ast}$.
\end{prop}

For $0<p\le\infty,$ we set\vspace{0.2cm}

$p^\prime=
\begin{cases}
\frac p{p-1}\quad &{\rm if}\ 1<p<\infty,\\
1 &{\rm if}\ p=\infty,\\
\infty &{\rm if}\ 0<p\le 1.
\end{cases}
$\vspace{0.2cm}

We are now in position to state our main results.
\begin{theorem}\label{T1}
Let $U$ be a $d$-set, $0 \le d \le n$, and $w_{i,j}(x)=(1+2^j\,{\rm
dist}(x,U))^{s_i^\prime},\,i=1,2$. Further, let
$t=\min(p_1^\prime,p_2),\, s^\prime=s_1^\prime-s_2^\prime>0$\ and\
$\frac 1{\tilde{p}}=\min\big(\frac\delta d,\frac
{s^\prime}n\big)+\frac 1{p_1}$. We assume that
 $0< p_1\le p_2\le \infty$\,\,or\,\,$\tilde{p}<p_2 < p_1\le
\infty$.

Denote by $a_k$ the $k$th approximation number of the embedding
(\ref{BB}). Then $a_{k}\sim k^{-\varkappa},$\ where\vspace{-0.2cm}
\beg{enumerate}
\item[$(i)$] $\varkappa=\min\big(\frac\delta d,\frac {s^\prime}n\big)$
\, if\, $0< p_1\le p_2\le 2$\,\,or\,\,$2\le p_1 \le p_2\le
\infty,$\vspace{-0.2cm}

\item[$(ii)$] $\varkappa =\min\big(\frac\delta d,\frac {s^\prime}n\big)
+\frac 1{p_1}-\frac 1{p_2}$\, if\,
$\tilde{p}<p_2<p_1\leq\infty$,\vspace{-0.2cm}

\item[$(iii)$] $\varkappa=\min\big(\frac\delta d,\frac {s^\prime}n\big)
 +\frac 12-\frac 1t$\,
 if\, $0< p_1 < 2 < p_2\le \infty$\,and \,$\min\big(\frac\delta d,\frac
{s^\prime}n\big)>\frac 1t$,\vspace{-0.2cm}

\item[$(iv)$] $\varkappa=\frac{s^\prime}n\cdot\frac t2$
\, if\, $0< p_1 < 2 < p_2\le \infty$ and $\delta>s^\prime,$ with the
following restrictions,

 $
\begin{cases}
\delta<\frac dt,\\ \delta-s^\prime<\frac{2d-n}t,
\end{cases}\ \ {\rm or}\ \
\begin{cases}
\delta+s^\prime<\frac nt,\\ \delta-s^\prime>\frac{2d-n}t,
\end{cases} $
 \end{enumerate}

Besides, suppose that in addition,  $0< p_1 < 2 < p_2\le \infty$ and
$\delta<s^\prime$. Then\vspace{-0.2cm} \beg{enumerate}
\item[$(i)$]
$k^{-\frac t2\cdot\min(\frac\delta d,\frac {s^\prime}n)}\preceq
a_{k}\preceq k^{-\frac t2\cdot\min(\frac\delta d,\frac
{\delta+s^\prime}{2n})}$\ if\ $\delta<\frac dt,\ s^\prime<\frac nt$
and $\delta-s^\prime<\frac{2d-n}t$;\vspace{-0.2cm}

\item[$(ii)$] $k^{-\frac t2\cdot\min(\frac\delta d,\frac
{s^\prime}n)}\preceq a_k\preceq k^{-\frac tn\cdot\min(\delta,\frac
{\delta+s^\prime}4)}$\ if\ $\delta+ s^\prime<\frac nt$ and
$\delta-s^\prime>\frac{2d-n}t$.
 \end{enumerate}
\end{theorem}

\begin{re}
Note that in the above assertion point (iv), as well as the latter
statements (i) and (ii), vanishes if\ $0< p_1 \le 1$ and
$p_2=\infty$.
\end{re}

\begin{theorem}\label{T2}
Let $U$ be a $d$-set, $0 \le d \le n$, and $w_{i,j}(x)=(1+2^j\,{\rm
dist}(x,U))^{s_i^\prime},\,i=1,2$. Further, let \vspace{-0.2cm}
 $$s^\prime=s_1^\prime-s_2^\prime>0,\ \ \theta = \frac{1/{p_1}-1/{p_2}}{1/2-1/{p_2}}\quad  and \quad\
\frac 1{\tilde{p}}=\min\big(\frac\delta d,\frac
{s^\prime}n\big)+\frac 1{p_1}.$$ We assume that  $0< p_1\le p_2\le
\infty$\,\,or\,\,$\tilde{p}<p_2 < p_1\le \infty$.

Denote by $d_k$ the $k$th Kolmogorov number of the embedding
(\ref{BB}). Then $d_{k}\sim k^{-\varkappa},$\ where \beg{enumerate}
\item[$(i)$] $\varkappa=\min\big(\frac\delta d,\frac {s^\prime}n\big)$
\, if\, $0< p_1\le p_2\le 2$\,\,or\,\,$2<p_1 = p_2\le \infty,$
\vspace{-0.2cm}
\item[$(ii)$] $\varkappa =\min\big(\frac\delta d,\frac {s^\prime}n\big)
+\frac 1{p_1}-\frac 1{p_2}$\, if\, $\tilde{p}<p_2<p_1\leq\infty$,
\vspace{-0.2cm}
\item[$(iii)$] $\varkappa=\min\big(\frac\delta d,\frac {s^\prime}n\big)
 +\frac 12-\frac 1{p_2}$\,
 if\, $0< p_1 < 2 < p_2\le \infty$\,and \,$\min\big(\frac\delta d,\frac
{s^\prime}n\big)>\frac 1{p_2}$, \vspace{-0.2cm}
\item[$(iv)$] $\varkappa=\frac{s^\prime}n\cdot\frac{p_2}2$
\, if\, $0< p_1 < 2 < p_2< \infty$ and $\delta>s^\prime,$ with the
following restrictions,
\\ $
\begin{cases}
\delta<\frac d{p_2},\\ \delta-s^\prime<\frac{2d-n}{p_2},
\end{cases}\ \ {\rm or}\ \
\begin{cases}
\delta+s^\prime<\frac n{p_2},\\
\delta-s^\prime>\frac{2d-n}{p_2}.
\end{cases} $\vspace{-0.2cm}

\item[$(v)$] $\varkappa=\min\big(\frac\delta d,\frac {s^\prime}n\big)
 +\frac 1{p_1}-\frac 1{p_2}$\,
 if\, $ 2\le p_1 < p_2\le \infty$\,and \,$\min\big(\frac\delta d,\frac
{s^\prime}n\big)>\frac \theta{p_2}$,\vspace{-0.2cm}

\item[$(vi)$] $\varkappa=\frac{s^\prime}n\cdot\frac{p_2}2$
\, if\, $2\le p_1<p_2< \infty$ and $\delta>s^\prime,$ with the
following restrictions,
\\ $
\begin{cases}
\delta<\frac d{p_2}\theta,\\
\delta-s^\prime<\frac{2d-n}{p_2}\theta,
\end{cases}\ \ {\rm or}\ \
\begin{cases}
\delta+s^\prime<\frac n{p_2}\theta,\\
\delta-s^\prime>\frac{2d-n}{p_2}\theta.
\end{cases} $
 \end{enumerate}

Besides, we have the following statements.\vspace{-0.2cm}

 \beg{enumerate}
\item[$(i)$] Suppose that in addition, $0< p_1 < 2 < p_2< \infty$
and $\delta<s^\prime$. Then

$(a)$\ $k^{-\frac {p_2}2\cdot\min(\frac\delta d,\frac
{s^\prime}n)}\preceq d_{k}\preceq k^{-\frac
{p_2}2\cdot\min(\frac\delta d,\frac {\delta+s^\prime}{2n})}$\ if\
$\delta<\frac d{p_2},\ s^\prime<\frac n{p_2}$ and
$\delta-s^\prime<\frac{2d-n}{p_2}$;

$(b)$\ $k^{-\frac {p_2}2\cdot\min(\frac\delta d,\frac
{s^\prime}n)}\preceq d_k\preceq
k^{-\frac{p_2}n\cdot\min(\delta,\frac {\delta+s^\prime}4)}$\ if\
$\delta+ s^\prime<\frac n{p_2}$ and
$\delta-s^\prime>\frac{2d-n}{p_2}$.
\item[$(ii)$] Suppose that in addition, $2\le p_1 < p_2< \infty$
and $\delta<s^\prime$. Then

$(a)$\ $k^{-\frac {p_2}2\cdot\min(\frac\delta d,\frac
{s^\prime}n)}\preceq d_{k}\preceq k^{-\frac
{p_2}2\cdot\min(\frac\delta d,\frac {\delta+s^\prime}{2n})}$\ if\
$\delta<\frac d{p_2}\theta,\ s^\prime<\frac n{p_2}\theta$ and
$\delta-s^\prime<\frac{2d-n}{p_2}\theta$;

$(b)$\ $k^{-\frac {p_2}2\cdot\min(\frac\delta d,\frac
{s^\prime}n)}\preceq d_k\preceq
k^{-\frac{p_2}n\cdot\min(\delta,\frac {\delta+s^\prime}4)}$\ if\
$\delta+ s^\prime<\frac n{p_2}\theta$ and
$\delta-s^\prime>\frac{2d-n}{p_2}\theta$.
 \end{enumerate}
\end{theorem}

\begin{re}
Points (iv) and (vi), as well as the latter statements (i) and (ii),
vanish if\ $p_2=\infty$.
\end{re}

\begin{theorem}\label{T3}
Let $U$ be a $d$-set, $0 \le d \le n$, and $w_{i,j}(x)=(1+2^j\,{\rm
dist}(x,U))^{s_i^\prime},\,i=1,2$. Further, let \vspace{-0.2cm}
 $$s^\prime=s_1^\prime-s_2^\prime>0,\ \ \theta_1 =
\frac{1/{p_2^\prime}-1/{p_1^\prime}}{1/2-1/{p_1^\prime}}\quad  and
\quad\ \frac 1{\tilde{p}}=\min\big(\frac\delta d,\frac
{s^\prime}n\big)+\frac 1{p_1}.$$ We assume that  $0< p_1\le p_2\le
\infty$\,\,or\,\,$\tilde{p}<p_2 < p_1\le \infty$.

Denote by $c_k$ the $k$th  Gelfand number of the embedding
(\ref{BB}). Then $c_{k}\sim k^{-\varkappa},$\ where
 \beg{enumerate}
\item[$(i)$] $\varkappa=\min\big(\frac\delta d,\frac {s^\prime}n\big)$
\, if\, $2\le p_1 \le p_2\le \infty$ \,\,or\,\ $0< p_1= p_2<
2$,\vspace{-0.2cm}

\item[$(ii)$] $\varkappa =\min\big(\frac\delta d,\frac {s^\prime}n\big)
+\frac 1{p_1}-\frac 1{p_2}$\, if\,
$\tilde{p}<p_2<p_1\leq\infty$,\vspace{-0.2cm}

\item[$(iii)$] $\varkappa=\min\big(\frac\delta d,\frac {s^\prime}n\big)
 +\frac 1{p_1}-\frac 12$\,
 if\, $0< p_1 < 2 < p_2\le \infty$\,and \,$\min\big(\frac\delta d,\frac
{s^\prime}n\big)>\frac 1{p_1^\prime}$,\vspace{-0.2cm}

\item[$(iv)$] $\varkappa=\frac{s^\prime}n\cdot\frac{p_1^\prime}2$
\, if\, $1< p_1 < 2 < p_2\le \infty$ and $\delta>s^\prime,$ with the
following restrictions,
\\ $
\begin{cases}
\delta<\frac d{p_1^\prime},\\
\delta-s^\prime<\frac{2d-n}{p_1^\prime},
\end{cases}\ \ {\rm or}\ \
\begin{cases}
\delta+s^\prime<\frac n{p_1^\prime},\\
\delta-s^\prime>\frac{2d-n}{p_1^\prime}.
\end{cases} $\vspace{-0.2cm}

\item[$(v)$] $\varkappa=\min\big(\frac\delta d,\frac {s^\prime}n\big)
 +\frac 1{p_1}-\frac 1{p_2}$\,
 if\, $0< p_1 < p_2\le 2$\,and \,$\min\big(\frac\delta d,\frac
{s^\prime}n\big)>\frac {\theta_1}{p_1^\prime}$,\vspace{-0.2cm}

\item[$(vi)$] $\varkappa=\frac{s^\prime}n\cdot\frac{p_1^\prime}2$
\, if\, $1< p_1 < p_2\le 2$ and $\delta>s^\prime,$ with the
following restrictions,
\\ $
\begin{cases}
\delta<\frac d{p_1^\prime}\theta_1,\\
\delta-s^\prime<\frac{2d-n}{p_1^\prime}\theta_1,
\end{cases}\ \ {\rm or}\ \
\begin{cases}
\delta+s^\prime<\frac n{p_1^\prime}\theta_1,\\
\delta-s^\prime>\frac{2d-n}{p_1^\prime}\theta_1.
\end{cases} $
 \end{enumerate}

Besides,  we have the following statements.\vspace{-0.2cm}
\beg{enumerate}
\item[$(i)$] Suppose that in addition, \, $1<
p_1 < 2 < p_2\le \infty$ and $\delta<s^\prime$. Then

$(a)$\  $k^{-\frac {p_1^\prime}2\cdot\min(\frac\delta d,\frac
{s^\prime}n)}\preceq c_{k}\preceq k^{-\frac
{p_1^\prime}2\cdot\min(\frac\delta d,\frac {\delta+s^\prime}{2n})}$\
if\ $\delta<\frac d{p_1^\prime},\ s^\prime<\frac n{p_1^\prime}$ and
$\delta-s^\prime<\frac{2d-n}{p_1^\prime}$,\vspace{-0.2cm}

$(b)$\  $k^{-\frac {p_1^\prime}2\cdot\min(\frac\delta d,\frac
{s^\prime}n)}\preceq c_k\preceq k^{-\frac
{p_1^\prime}n\cdot\min(\delta,\frac {\delta+s^\prime}4)}$\ if\
$\delta+ s^\prime<\frac n{p_1^\prime}$ and
$\delta-s^\prime>\frac{2d-n}{p_1^\prime}$.\vspace{-0.2cm}

\item[$(ii)$] Suppose that in addition,\ $1< p_1 < p_2\le 2$
and $\delta<s^\prime$. Then

$(a)$\ $k^{-\frac {p_1^\prime}2\cdot\min(\frac\delta d,\frac
{s^\prime}n)}\preceq c_{k}\preceq k^{-\frac
{p_1^\prime}2\cdot\min(\frac\delta d,\frac {\delta+s^\prime}{2n})}$\
if $\delta<\frac d{p_1^\prime}\theta_1,\ s^\prime<\frac
n{p_1^\prime}\theta_1$ and
$\delta-s^\prime<\frac{2d-n}{p_1^\prime}\theta_1$,\vspace{-0.2cm}

$(b)$\ $k^{-\frac {p_1^\prime}2\cdot\min(\frac\delta d,\frac
{s^\prime}n)}\preceq c_k\preceq
k^{-\frac{p_1^\prime}n\cdot\min(\delta,\frac {\delta+s^\prime}4)}$\
if\ $\delta+ s^\prime<\frac n{p_1^\prime}\theta_1$ and
$\delta-s^\prime>\frac{2d-n}{p_1^\prime}\theta_1$.
 \end{enumerate}
\end{theorem}

\begin{re}
Points (iv) and (vi), together with the latter statements (i) and
(ii), vanish if\ \ $0<p_1\le 1$.
\end{re}

\begin{re}
We shift the proofs of the above three theorems to Section
\ref{wmb}.
\end{re}

Now, we wish to compare the approximation, Gelfand and Kolmogorov
numbers of the embedding (\ref{BB}).  The comparison of these above
results shows that\vspace{-0.2cm}
\begin{enumerate}
\item[$(i)$]\ $a_n\sim c_n$ if either

$(a)$\ $2\le p_1< p_2\le \infty$\ or,

$(b)$\ $\tilde{p}<p_2\le p_1\le \infty$\ or,

$(c)$\ $1\le p_1 < p_1^\prime\le p_2\le\infty$\ and \
$\min\big(\frac\delta d,\frac {s^\prime}n\big)>\frac 1{p_1^\prime}$;
\vspace{-0.2cm}

\item[$(ii)$]\ $a_n\sim d_n$ if either

$(a)$\ $0< p_1< p_2\le 2$\ or,

$(b)$\ $\tilde{p}<p_2\le p_1\le \infty$\ or,

$(c)$\ $0< p_1 < 2 < p_2\le p_1^\prime\le\infty$\ and\
$\min\big(\frac\delta d,\frac {s^\prime}n\big)>\frac
1{p_2}$;\vspace{-0.2cm}

\item[$(iii)$]\ $c_n\sim d_n$ if either

$(a)$\ $\tilde{p}<p_2\le p_1\le \infty$\ or,

$(b)$\ $1\le p_1 < p_1^\prime= p_2\le\infty$\ and\
$\min\big(\frac\delta d,\frac {s^\prime}n\big)>\frac 1{p_2}$.
\end{enumerate}
Note that we don't discuss above the case when $0< p_1 < 2 < p_2\le
\infty$ and $\min\big(\frac\delta d,\frac {s^\prime}n\big)<\frac
1{\min(p_1^\prime,p_2)}$.

\section{Widths in sequence spaces}\label{snss}

This section is the heart of the paper. Our main aim will be to
determine the asymptotic behaviour of related widths of compact
embeddings between weighted sequence spaces
$\ell_q(2^{js}\ell_p(w))$, where the sequences are indexed by
$\mathbb{N}_0\times\mathbb{Z}^n$.

\subsection{Preliminaries}\label{prel}

Here we are going to use the discrete wavelet transform in order to
obtain equivalent quasi-norms in the spaces $
B_{p,q}^{s,s^\prime}(\mathbb{R}^n, U)$ which, in a quite natural
way, will establish isomorphism between $
B_{p,q}^{s,s^\prime}(\mathbb{R}^n, U)$ and appropriate sequence
spaces, cf. \cite{Ke08}. By now this is a standard method to reduce
complicated problems in function spaces to simpler problems in
sequence spaces. The key point in this discretization technique is
that the asymptotic order of the estimates is preserved.

Wavelet bases in function spaces are a well-developed concept, see
\cite{Tr08} for a survey. We adopt the notation from \cite{Tr06}
(Section 4.2.1) with $l = 0$. Let $\psi_M, \psi_F\in
C^k(\mathbb{R})$ be real compactly supported Daubechies wavelets
with
$$
\int_\mathbb{R}x^\beta \psi_M(x)dx=0\quad \quad {\rm for}\ \
|\beta|<k.
$$
Let $G^0=\{F,M\}^n$ and let $ G^j=\{F,M\}^{n^\ast}$ where $n^\ast$
indicates that at least one $G_i$ of $G=(G_1,\ldots,G_n)\in
\{F,M\}^{n^\ast}$ must be an $M$. It is clear that the cardinal
number of $\{F,M\}^{n^\ast}$ is $2^n-1$. Let for $x\in\mathbb{R}^n$
$$
\psi_{Gm}^j(x)=2^{j\frac n2}\prod\limits_{i=1}^n
\psi_{G_i}(2^jx_i-m_i)\ \ \ {\rm where}\ \ j\in\mathbb{N}_0,
m\in\mathbb{Z}^n\ {\rm and}\ G=(G_1,\ldots,G_n)\in G^j.
$$
Then $\{\psi_{Gm}^j: \ j\in\mathbb{N}_0,\ G\in G^j,\
m\in\mathbb{Z}^n\}$ is an orthonormal basis in $L_2(\mathbb{R}^n)$,
see  \cite{Ke08,Tr06}.

In our situation we shall consider the following sequence spaces,
see \cite{Ke08} (Section 5.3.4).
\begin{definition}
Let $w=(w_j)_{j\in\mathbb{N}_0}$ be as in (\ref{w_U}). We put
$$\tilde{b}_{p,q}^{s,s^\prime}(w)=\Big\{(\lambda_{Gm}^j)_{j,G,m}\,:\ \|\lambda\, |\,
\tilde{b}_{p,q}^{s,s^\prime}(w)\|<\infty\Big\},
$$
where
$$
\|\lambda\, |\, \tilde{b}_{p,q}^{s,s^\prime}(w)\|= \Bigg(
\sum\limits_{j=0}^\infty 2^{j(s-\frac np)q}\sum\limits_{G\in G^j}
\Big(\sum\limits_{m\in\mathbb{Z}^n}|\lambda_{Gm}^j|^pw_j^p(2^{-j}m)
\Big)^{q/p} \Bigg)^{1/q}.
$$
\end{definition}\label{bws}
The following proposition can also be found there, see \cite{Ke08}
(Corollary 5.33); cf. also \cite{Ke10,LS11}.
\begin{prop}\label{iso}
Let $U\in \mathbb{R}^n$ bounded  and $w=(w_j)_{j\in\mathbb{N}_0}$ as
in (\ref{w_U}). Further, let $f\in
B^{s,s^\prime}_{p,q}(\mathbb{R}^n,U)$, k large enough and
$$\lambda_{Gm}^j=\lambda_{Gm}^j(f)=2^{j\frac n2}\langle f,\psi_{Gm}^j\rangle
=2^{j\frac n2}\int f(x)\psi_{Gm}^jdx.$$ Then
$$I:\ f\mapsto 2^{j\frac n2}\langle f,\psi_{Gm}^j\rangle
$$
is an isomorphic map from $B^{s,s^\prime}_{p,q}(\mathbb{R}^n,U)$
onto $\tilde{b}_{p,q}^{s,s^\prime}(w)$.
\end{prop}

Inspired by Proposition \ref{iso}, we shall work with the following
weighted sequence spaces. Let $(w_j)_{j\in \mathbb{N}_0}$ be a given
weight sequence as in (\ref{w_U}). We put
\begin{equation}\label{ellwU}
\begin{split}
\ell_q(2^{js}\ell_p(w)):=\Bigg\{&
\lambda=(\lambda_{j,m})_{j,m}:~~\lambda_{j,m}\in\mathbb{C},\\
&\|\lambda|\ell_q(2^{js}\ell_p(w))\|= \Big( \sum\limits_{j=0}^\infty
2^{jsq}
\Big(\sum\limits_{m\in{\mathbb{Z}}^n}|\lambda_{j,m}\,w_j(2^{-j}m)|^p
\Big)^{q/p}\Big)^{1/q}<\infty \Bigg\},
\end{split}
\end{equation}
(usual modification if $p=\infty$ and/or $q=\infty$). If $s=0$ we
will write $\ell_q(\ell_p(w))$. In contrast to the quasi-norm
defined in Definition \ref{bws}, the finite summation on $G\in G^j$
is irrelevant and can be omitted. Similar considerations may be
found in \cite{HS11,KLSS06b,LS11,Sk05}. Furthermore, let
\begin{equation*}
\begin{array}{ll}
A_1:=\ell_{q_1}(2^{j(s_1-\frac n{p_1})}\ell_{p_1}(v_1)),\quad\quad &
A_2:=\ell_{q_2}(2^{j(s_2-\frac n{p_2})}\ell_{p_2}(v_2)),\vspace{0.2cm}\\
B_1:=\ell_{q_1}(2^{j\delta}\ell_{p_1}(w)),\quad\quad &
B_2:=\ell_{q_2}(\ell_{p_2}),
\end{array}
\end{equation*}
where $$ v_1=(w_{1,j})_{j\in\mathbb{N}_0},\
v_2=(w_{2,j})_{j\in\mathbb{N}_0}\ \ {\rm and}\ \
w=(w_j)_{j\in\mathbb{N}_0}\ \ {\rm with}\ \
w_j(x)=\frac{w_{1,j}(x)}{w_{2,j}(x)},\ \ {\rm for}\ \
x\in\mathbb{Z}^n.$$ As is already discussed in \cite{LS11}, we
observe by the properties of $s$-numbers that
\begin{equation}\label{AABB}
s_k({\rm id}, A_1, A_2)=s_k({\rm id}, B_1, B_2),\quad \quad
k=1,2,\ldots,
\end{equation}
where $s_k$ denotes any of the three quantities\ $a_k,\, d_k$ or
$c_k$.  Hence, we may concentrate on $s_k({\rm id}, B_1, B_2)$.

We shall consider compact subsets of $\mathbb{R}^n$ for $U$, in
order to guarantee the compactness of the embedding
$B_1\hookrightarrow B_2$. We are now able to recall the definition
of d-sets, which are fractal sets in between single point sets
$\{x_0\}$ and compact sets with non-empty interior, cf. \cite{Tr98}.

Let $U$ be a compact set and $\mu$ a Radon measure with ${\rm supp}
\mu = U$. The set $U$ is called a $d$-set, $0 \le d \le n$, if for
each ball of radius $r$ and centered in $\gamma\in U$ holds
$$
\mu(B(\gamma,r))\sim r^d,\ \ {\rm with}\ \ 0<r<1.
$$

Let us mention the estimation of a number of dyadic cubes of a fixed
side length that are in a predetermined distance to the set $U$.

 For $i,\, j\in\mathbb{N}_0$, we denote by $N_{j,i}$ the number of cubes
$Q_{j,\ell}$ of side length $2^{-j}$, centered in $2^{-j}\ell$ with
\begin{equation}
\sqrt{n}2^{-j+i}<{\rm dist} (Q_{j,\ell}, U)\le 4\sqrt{n}2^{-j+i}, \
\ \ell\in\mathbb{Z}^n.
\end{equation}

The following lemma may be found in \cite{LS11}.
\begin{lemma}\label{Nji}
Let $U$ be a $d$-set, then
$$
N_{j,i}\sim
\begin{cases}
2^{in}2^{(j-i)d}\ \ \ & 0\le i<j,
\\
2^{in}&j\le i.
\end{cases}
$$
\end{lemma}

Following Pietsch \cite{Pie87}, we associate to the sequence of the
$s$-numbers the following operator ideals, and for $0<r<\infty$, we
put
\begin{equation}\mathscr{L}_{r,\infty}^{(s)}:=\left\{T\in\mathcal{L}(X,
Y):\quad \sup\limits_{n\in\mathbb{N}}n^{1/r}s_n(T)<\infty\right\}.
\end{equation}
 Equipped with the quasi-norm
\begin{equation}\label{idealsdef}
L_{r,\infty}^{(s)}(T):=\sup\limits_{n\in\mathbb{N}}n^{1/r}s_n(T),
\end{equation}
the set $\mathscr{L}_{r,\infty}^{(s)}$ becomes a quasi-Banach space.
For such quasi-Banach spaces there always exists a real number
$0<\rho\leq 1$ such that
\begin{equation}\label{idealsinq}
L_{r,\infty}^{(s)}\left(\sum\limits_jT_j\right)^\rho\leq
\sum\limits_jL_{r,\infty}^{(s)}(T_j)^\rho
\end{equation}
holds for any sequence of operators
$T_j\in\mathscr{L}_{r,\infty}^{(s)}.$ Then we shall use the
quasi-norms $L_{r,\infty}^{(a)},~L_{r,\infty}^{(c)}$ and
$L_{r,\infty}^{(d)}$ for the approximation, Gelfand and Kolmogorov
numbers, respectively.

\begin{re}\label{opt}
We would like to add some comments on operator ideals. Historically,
the technique of estimating single s-numbers (or entropy numbers)
via estimates of ideal (quasi-)norms derives from ideas of Carl
\cite{Car81}. In the 1980s this technique was frequently used in
operator theory, in eigenvalue problems for Banach space operators,
etc. In the function spaces community however, the operator ideal
technique remained unknown for many years. As far as we know, it was
applied for the first time in \cite{Ku03,KLSS03a}, both of which
appeared in 2003.
\end{re}

For brevity's sake, we wish to make an additional agreement
throughout the following three subsections. Let $U$ be a $d$-set, $0
\le d \le n$, and $w_{j,\ell}=w_j(2^{-j}\ell)=(1+2^j\,{\rm
dist}(2^{-j}\ell,U))^{s^\prime}$ a sequence of weights,
$j\in\mathbb{N}_0,\, \ell\in\mathbb{Z}^n$,
$s^\prime=s_1^\prime-s_2^\prime>0$, if no further restrcitions are
stated.

\subsection{Approximation numbers of sequence spaces}\label{ans}
To begin with, we shall recall some lemmata. Lemma \ref{an1} follows
trivially from results of Gluskin \cite{Gl83} and Edmunds and
Triebel \cite{ET96}. Lemma \ref{an1inf} is due to Vyb\'iral
\cite{Vy08}.
\begin{lemma}\label{an1}
 Let $N\in\mathbb{N}$ and $k\le\frac{N}{4}$.
\vspace{-0.2cm} \beg{enumerate}
\item[$(i)$] If $0< p_1\le p_2\le 2$~or~$2\le p_1\le p_2\le \infty$~ then
$$a_{k}\big({\rm id}, \ell_{p_1}^N, \ell_{p_2}^N\big)\sim 1.$$
\vspace{-0.8cm}
\item[$(ii)$] If $1\le p_1 < 2 < p_2\le \infty$ and $(p_1, p_2) \neq (1,
\infty)$\, then
$$a_{k}\big({\rm id}, \ell_{p_1}^N, \ell_{p_2}^N\big)\sim \min\big(1, N^{1/t}k^{-1/2}\big).$$
where $\frac 1t= \frac 1{\min(p_1^\prime, p_2)}.$ \vspace{-0.2cm}

\item[$(iii)$] If
$0<  p_1< 1$ and $ 2 < p_2< \infty$\, then
$$a_{k}\big({\rm id}, \ell_{p_1}^N, \ell_{p_2}^N\big)\sim
\min\big(1, N^{1/{p_2}}k^{-1/2}\big).$$ \end{enumerate}
\end{lemma}

\begin{lemma}\label{an1inf}
Let $0 < p\le 1$ and $N\in\mathbb{N}$. \vspace{-0.2cm}
\beg{enumerate}
\item[$(i)$] Let $0<\lambda<1$. Then there exists a constant
$c_\lambda>0$ depending only on $\lambda$ such that
\begin{equation}\label{an1infupp}
a_{k}\big({\rm id}, \ell_{p}^N, \ell_\infty^N\big)\le
\begin{cases}
1 & {\rm if}\ \ k\leq N^\lambda,\\
c_\lambda k^{-1/2}\quad & {\rm if}\ \ N^\lambda<k\le N,\\
0 & {\rm if}\ \ k>N. \end{cases}
\end{equation}
\vspace{-0.6cm}
\item[$(ii)$] There exists a constant
$C>0$ independent of $k$ such that for any  $k\in \mathbb{N}$
\begin{equation}\label{an1inflow}
a_{k}\big({\rm id}, \ell_{p}^{2k}, \ell_\infty^{2k}\big)\ge C
n^{-1/2}.
\end{equation}
\end{enumerate}

\end{lemma}

Lemma \ref{an2} in the case $1\le p_2 < p_1\le \infty$ may be found
in Pietsch \cite{Pie78}, Section 11.11.5, also in Pinkus
\cite{Pin85}(p.\,203). The proof may be directly generalized to the
quasi-Banach setting $0< p_2 < p_1\le \infty.$

\begin{lemma}\label{an2}
Let $0<  p_2 < p_1\leq \infty$ and $k\le N$. Then
$$a_{k}\left({\rm id}, \ell_{p_1}^N, \ell_{p_2}^N\right)
=(N - k + 1)^{1/{p_2}-1/{p_1}}.$$
\end{lemma}

The following lemma in the case $1\le p_1 < 2 < p_2\le \infty$ may
be found in  \cite{Sk05}. The proof may be trivially extended to the
quasi-Banach spaces with $0< p_1 < 2 < p_2\le \infty$, by virtue of
Lemma \ref{an1} (iii).

\begin{lemma}\label{an3}  Suppose $0< p_1 < 2 < p_2\le \infty$,
and assume that $p_2\neq\infty$\ when\ $0<p_1\le 1$. Then there is a
positive constant $C$ independent of $N$ and $k$ such that
\begin{equation}\label{alp2p}
a_{k}\Big({\rm id}, ~\ell_{p_1}^N, ~\ell_{p_2}^N\Big)\leq C\,
\left\{
\begin{array}{ll}
1 & {\rm if}\ \ k\leq N^{2/t},\\
N^{1/t}k^{-1/2}\quad & {\rm if}\ \ N^{2/t}< k \leq
N,\\
0 & {\rm if}\ \ n>N,
\end{array}\right.
\end{equation}
where $\frac 1t= \frac 1{\min(p_1^\prime, p_2)}.$
\end{lemma}

\begin{prop}\label{an4}
Suppose $0< p_1 < 2 < p_2 \le \infty$\,and $
t=\min(p_1^\prime,p_2)$.
 \beg{enumerate}
\item[$(i)$] If\, $\min\big(\frac\delta d,\frac
{s^\prime}n\big)>\frac 1t$, then
 \vspace{-0.2cm}
\begin{equation}
a_{k}\Big({\rm id}, \ell_{q_1}(2^{j\delta}\ell_{p_1}(w)),
\ell_{q_2}(\ell_{p_2})\Big) \sim k^{-\min\big(\frac\delta d,\frac
{s^\prime}n\big)
 +\frac 1t-\frac 12}.
\end{equation}
 \vspace{-0.4cm}
\item[$(ii)$] If $\delta > s^\prime$ and either
 $
\begin{cases}
\delta<\frac dt,\\ \delta-s^\prime<\frac{2d-n}t,
\end{cases}\ \ {\rm or}\ \
\begin{cases}
\delta+s^\prime<\frac nt,\\ \delta-s^\prime>\frac{2d-n}t,
\end{cases} $
then
\begin{equation}
a_{k}\Big({\rm id}, \ell_{q_1}(2^{j\delta}\ell_{p_1}(w)),
\ell_{q_2}(\ell_{p_2})\Big) \sim k^{-\frac {s^\prime}n\cdot\frac
t2}.
\end{equation}
 \vspace{-0.4cm}
\item[$(iii)$] If\ $\delta<\frac dt,\ \delta<s^\prime<\frac nt$
and $\delta-s^\prime<\frac{2d-n}t$, then \vspace{-0.2cm}
\begin{equation}k^{-\frac
t2\cdot\min(\frac\delta d,\frac {s^\prime}n)}\preceq a_{k}\Big({\rm
id}, \ell_{q_1}(2^{j\delta}\ell_{p_1}(w)),
\ell_{q_2}(\ell_{p_2})\Big)\preceq k^{-\frac t2
\cdot\min(\frac\delta d,\frac {\delta+s^\prime}{2n})}.
\end{equation}
 \vspace{-0.4cm}
\item[$(iv)$] If\ $\delta<s^\prime,\ \delta+ s^\prime<\frac nt$ and
$\delta-s^\prime>\frac{2d-n}t$, then \vspace{-0.2cm}
\begin{equation}k^{-\frac
t2\cdot\min(\frac\delta d,\frac {s^\prime}n)}\preceq a_k\Big({\rm
id}, \ell_{q_1}(2^{j\delta}\ell_{p_1}(w)),
\ell_{q_2}(\ell_{p_2})\Big)\preceq k^{-\frac
tn\cdot\min(\delta,\frac {\delta+s^\prime}4)}.
\end{equation}
 \end{enumerate}
\end{prop}
\begin{proof}
{\tt Step 1}.\quad Preparations. We denote
$$\Lambda:=\{\lambda=(\lambda_{j,\ell}):
\,\quad\lambda_{j,\ell}\in\mathbb{C},\,\quad j\in
\mathbb{N}_0,\,\ell\in\mathbb{Z}^n\},$$ and set
$$ B_1=\ell_{q_1}(2^{j\delta}\ell_{p_1}(w))\,\quad {\rm
and}\,\quad  B_2= \ell_{q_2}(\ell_{p_2}).$$
Let\,$I_{j,i}\subset\mathbb{N}_0\times\mathbb{Z}^n$ be such that

\begin{equation}\label{Ij0}
I_{j,0}:=\{(j,\ell):\, {\rm dist} (2^{-j}\ell, U)\leq
\sqrt{n}2^{-j}\},\quad j\in\mathbb{N}_0,
\end{equation}
\begin{equation}\label{Iji}
I_{j,i}:=\{(j,\ell):\, \sqrt{n}2^{-j+i-1}<{\rm dist} (2^{-j}\ell,
U)\leq \sqrt{n}2^{-j+i}\}, \quad i\in \mathbb{N},\quad
j\in\mathbb{N}_0.
\end{equation}

Besides, let $P_{j,i}:\Lambda\mapsto\Lambda$ be the canonical
projection with respect to $I_{j,i}$, i.e., for $\lambda\in\Lambda$,
we put
\begin{equation*}
(P_{j,i}\lambda)_{u,v}:= \left\{
\begin{array}{ll}
\lambda_{u,v}\quad & (u,v)\in I_{j,i},\\
0\quad & {\rm otherwise},
\end{array}\right. \quad u\in
\mathbb{N}_0,\quad v\in\mathbb{Z}^n,\quad i\ge 0.
\end{equation*}
Then
\begin{equation}\label{id}
{\rm id}_\Lambda= \sum\limits_{j=0}^\infty\sum\limits_{i=0}^\infty
P_{j,i}.
\end{equation}

\begin{equation}
w_j(2^{-j}\ell) \sim (1+2^j2^{-j+i})^{s^\prime} \sim
2^{is^\prime}\quad {\rm if}\quad (j,\ell)\in I_{j,i}, \quad i \ge 0.
\end{equation}

Due to Lemma \ref{Nji} and the structure of $U$, cf. \cite{LS11}, we
have

\begin{equation}\label{rankpiece}
M_{j,i}:=|I_{j,i}|\le N_{j,i+2}+N_{j,i+3} \sim
\begin{cases}
2^{in}2^{(j-i)d},\ \ \ \ & 0\le i<j,\\
2^{in},& 0\le j\le i, \end{cases}
\end{equation}
and
\begin{equation}\label{rankpiecelow}
M_{j,i+1}+M_{j,i+2}\ge N_{j,i} \sim
\begin{cases}
2^{in}2^{(j-i)d},\ \ \ \ & 0\le i<j,\\
2^{in},& 0\le j\le i. \end{cases}
\end{equation}
 Thanks to simple monotonicity arguments
and explicit properties of the approximation numbers, there is a
positive constant $C$ independent of $k,\,j$ and $i$ such that
\begin{equation}\label{andiscr}
a_k(P_{j,i},B_1,B_2)\leq C  2^{-j\delta}2^{-is^\prime}a_k({\rm id},
\ell_{p_1}^{M_{j,i}}, \ell_{p_2}^{M_{j,i}}).
\end{equation}
{\tt Step 2}.\quad The operator ideal plays an important role. To
shorten notations we shall put $\frac 1s=\frac 1r+\frac 12$\, for
any $s>0$. By (\ref{idealsdef}) and (\ref{andiscr}), we have
\begin{equation}\label{idealdiscr}
L_{s,\infty}^{(a)}(P_{j,i})\leq C
2^{-j\delta}2^{-is^\prime}L_{s,\infty}^{(a)}({\rm id},
\ell_{p_1}^{M_{j,i}}, \ell_{p_2}^{M_{j,i}})
\end{equation} The known asymptotic
behavior of the approximation numbers $a_k({\rm
id},\ell_{p_1}^{N},\ell_{p_2}^{N})$, cf. (\ref{alp2p}), and
(\ref{rankpiece}) yield that, with the assumption $p_2\neq \infty$
if $0<p_1\le1$, \vspace{-0.2cm}
\begin{equation}\label{idealse2}
 L_{2,\infty}^{(a)}({\rm id}, \ell_{p_1}^{M_{j,i}},
 \ell_{p_2}^{M_{j,i}})\leq
C
\begin{cases}
2^{(in+d(j-i))/t},\ \ \  & 0\le i<j,\\
2^{\frac {in}t}, & 0\le j\le i,
\end{cases}\indent\indent\indent\
\end{equation}
\begin{equation}\label{ideals2}
L_{s,\infty}^{(a)}({\rm id}, \ell_{p_1}^{M_{j,i}},
\ell_{p_2}^{M_{j,i}})\leq C
\begin{cases}
2^{(in+d(j-i))(\frac 1t+\frac 1r)},\ \ \ \ & 0\le i<j,\ \frac 1s>\frac 12,\\
2^{in(\frac 1t+\frac 1r)},& 0\le j\le i,\ \frac 1s>\frac 12,
\end{cases}
\end{equation}
and in consequence \vspace{-0.2cm}
\begin{equation}\label{idealPji2}
L_{2,\infty}^{(a)}(P_{j,i})\leq C 2^{-j\delta}2^{-is^\prime}
\begin{cases}2^{(in+d(j-i))/t},\ \ \  & 0\le i<j, \\
2^{\frac {in}t}, & 0\le j\le i,
\end{cases}\indent\indent\indent\
\end{equation}
 \vspace{-0.1cm}
\begin{equation}\label{idealPjis}
L_{s,\infty}^{(a)}(P_{j,i})\leq C 2^{-j\delta}2^{-is^\prime}
\begin{cases}
2^{(in+d(j-i))(\frac 1t+\frac 1r)},\ \ \ \ & 0\le i<j,\ \frac 1s>\frac 12,\\
2^{in(\frac 1t+\frac 1r)},& 0\le j\le i,\ \frac 1s>\frac 12.
\end{cases}
\end{equation}

If $0<p_1\le1$ and $p_2=\infty$, we have $t=\infty$ and select
$0<\lambda<1$ such that
$\frac\lambda{2(1-\lambda)}<\min\big(\frac\delta d,\frac
{s^\prime}n\big)$. The inequality $\lambda\cdot \frac 1s\le\frac
1s-\frac 12$ holds if and only if $\frac 1s\ge\frac
1{2(1-\lambda)}$, where $0<\lambda<1$. So there exists a constant
$s>0$ such that $\lambda\cdot \frac 1s<\frac 1s-\frac
12<\min\big(\frac\delta d,\frac {s^\prime}n\big)$. Then, in a
similar way as above, we find by (\ref{an1infupp}) that
\begin{equation}\label{idealPjislam}
L_{s,\infty}^{(a)}(P_{j,i})\leq C 2^{-j\delta}2^{-is^\prime}
\begin{cases}
2^{(in+d(j-i))\frac 1r},\ \ \ \ & 0\le i<j,\ \frac 1s>\frac
1{2(1-\lambda)},\\
2^{\frac{in}r},& 0\le j\le i,\ \frac 1s>\frac 1{2(1-\lambda)}.
\end{cases}
\end{equation}
 {\tt Step 3}.\quad The estimate of $a_k({\rm id},
B_1, B_2)$ from above in case (i), $\min\big(\frac\delta d,\frac
{s^\prime}n\big)>\frac 1t$. Let $M\in\mathbb{N}_0$ be given. For the
identity operator ${\rm id}_\Lambda$, we use the same division, as
in the proof of Theorem 6 in \cite{LS11},
\begin{equation}\label{PQ5}
\begin{split}
P^1:=\sum\limits_{j=0}^M\sum\limits_{i=0}^{j-1}P_{j,i},\quad\quad
&P^2:=\sum\limits_{j=M+1}^\infty\sum\limits_{i=0}^{j-1}P_{j,i},\\
Q^1:=\sum\limits_{j=0}^M\sum\limits_{i=j}^M P_{j,i},\quad\quad
 &
Q^2:=\sum\limits_{j=0}^M \sum\limits_{i=M+1}^\infty
P_{j,i},\quad\quad
Q^3:=\sum\limits_{j=M+1}^\infty\sum\limits_{i=j}^\infty P_{j,i}.
\end{split}
\end{equation}
Next, the proof in this case follows literally the presentation
given in \cite{LS11} with the obvious changes, now using $1/r+1/t$
instead of $1/r-1/p$, where the latter notations derive from
\cite{LS11}. So we don't expand here.\\
 {\tt Step 4}.\quad Now let $\delta\neq s^\prime,\ {s^\prime}<n/t$ and
 $\delta-s^\prime\neq\frac{2d-n}t$.
We turn to using the following division
\begin{equation}\label{PQ6}
\begin{split}
{\rm id}= &\ \ \
\sum\limits_{m=0}^{M_1}\sum\limits_{\stackrel{i+j=m}{i<j}}P_{j,i}+
\sum\limits_{m=M_1+1}^{M_2}\sum\limits_{\stackrel{i+j=m}{i<j}}P_{j,i}+
\sum\limits_{m=M_2+1}^\infty \sum\limits_{\stackrel{i+j=m}{i<j}}P_{j,i}\\
&+\sum\limits_{m=0}^{M_3}\sum\limits_{\stackrel{i+j=m}{i\ge
j}}P_{j,i}+
\sum\limits_{m=M_3+1}^{M_4}\sum\limits_{\stackrel{i+j=m}{i\ge
j}}P_{j,i}+ \sum\limits_{m=M_4+1}^\infty
\sum\limits_{\stackrel{i+j=m}{i\ge j}}P_{j,i},
\end{split}
\end{equation}
where $M_1, M_2, M_3, M_4\in\mathbb{N}$,\,\,$M_1<M_2$\, and\,
$M_3<M_4$, which will be determined later on for given
$k\in\mathbb{N}$. In terms of the subadditivity of $s$-numbers, we
observe
\begin{equation}\label{an_divi6}
a_{k^\prime}({\rm id}, B_1, B_2)\leq
\triangle_1+\triangle_2+\triangle_3+\triangle_4+\triangle_5+\triangle_6,
\end{equation}
where
\begin{equation*}
\begin{split}
\triangle_1=\sum\limits_{m=0}^{M_1}\sum\limits_{\stackrel{i+j=m}{i<j}}a_{k_{j,i}}
(P_{j,i}),&\ \
\triangle_2=\sum\limits_{m=M_1+1}^{M_2}\sum\limits_{\stackrel{i+j=m}{i<j}}a_{k_{j,i}}
(P_{j,i}),\ \
\triangle_3=\sum\limits_{m=M_2+1}^\infty \sum\limits_{\stackrel{i+j=m}{i<j}}\|P_{j,i}\|,\\
\triangle_4=\sum\limits_{m=0}^{M_3}\sum\limits_{\stackrel{i+j=m}{i\ge
j}}a_{k_{j,i}} (P_{j,i}),&\ \
\triangle_5=\sum\limits_{m=M_3+1}^{M_4}\sum\limits_{\stackrel{i+j=m}{i\ge
j}}a_{k_{j,i}} (P_{j,i}),\ \
\triangle_6=\sum\limits_{m=M_4+1}^\infty
\sum\limits_{\stackrel{i+j=m}{i\ge j}}\|P_{j,i}\|,
\end{split}
\end{equation*}
and\vspace{-0.2cm}
\begin{equation*}
k^\prime-1=\sum\limits_{m=0}^{M_2}\sum\limits_{\stackrel{i+j=m}{i<j}}(k_{j,i}-1)+
\sum\limits_{m=0}^{M_4}\sum\limits_{\stackrel{i+j=m}{i\ge
j}}(k_{j,i}-1).
\end{equation*}
{\tt Substep 4.1.}\quad First we deal with those parts concerning
$i<j$\, (i.e., $\triangle_1,\triangle_2$ and $\triangle_3$). We take
\begin{equation*}
M_1=\begin{cases} \Big[\frac {\log_2k}{n/2}\Big],\ \ &{\rm if}\
2d<n,\vspace{0.1cm}\\
\Big[\frac {\log_2k}d\Big],\ \ &{\rm if}\ 2d>n,\vspace{0.1cm}\\
 \Big[\frac
{\log_2k}d-\frac{\log_2\log_2k}d\Big], &{\rm if}\ 2d=n,
\end{cases}\ \ \ {\rm and}\ \ \
M_2=\begin{cases} \Big[\frac t2 \cdot\frac{\log_2k}d\Big],\ \ &{\rm
if}\
\delta-s^\prime<\frac{2d-n}t,\vspace{0.8cm}\\
\Big[\frac t2 \cdot\frac {\log_2k}{n/2}\Big],\ \ &{\rm if}\
\delta-s^\prime>\frac{2d-n}t,
\end{cases}
\end{equation*}
where $[a]$ denotes the largest integer smaller than
$a\in\mathbb{R}$ and $\log_2k$ is a dyadic logarithm of $k$. Then
\begin{equation*}
\begin{split}
\triangle_3&=\sum\limits_{m=M_2+1}^\infty
\sum\limits_{\stackrel{i+j=m}{i<j}}\|P_{j,i}\|\le
c_1\sum\limits_{m=M_2+1}^\infty
\sum\limits_{\stackrel{i+j=m}{i<j}} 2^{-j\delta}2^{-is^\prime}\\
&\le \begin{cases}
c_2 \sum\limits_{m=M_2+1}^\infty 2^{-m\delta}\le
c_3 2^{-M_2\delta},\ \ &{\rm if}\ \ \delta<s^\prime,\\
c_2 \sum\limits_{m=M_2+1}^\infty 2^{-\frac m2(\delta+s^\prime)}\le
c_3 2^{-M_2\frac {\delta+s^\prime}2},\ \ &{\rm if}\ \
\delta>s^\prime.
\end{cases}
\end{split}
\end{equation*}
Next, we choose proper $k_{j,i}$ for estimating $\triangle_1$ and
$\triangle_2$. If $i<j$ and $i+j\leq M_1,$ we take
$k_{j,i}=M_{j,i}+1$ such that $a_{k_{j,i}}(P_{j,i})=0$ and
$\triangle_1=0$. And we obtain
$$
\sum\limits_{m=0}^{M_1}\sum\limits_{\stackrel{i+j=m}{i<j}}k_{j,i}\leq
 c_1\sum\limits_{m=0}^{M_1}\sum\limits_{\stackrel{i+j=m}{i<j}}2^{dj}2^{(n-d)i}
\le \begin{cases} c_2\sum\limits_{m=0}^{M_1}2^{m\frac n2}\le
c_32^{M_1\frac n2}\leq c_3 k,\ \ &{\rm if}\
2d<n,\vspace{0.2cm}\\
c_2\sum\limits_{m=0}^{M_1}2^{md}\le
c_32^{M_1d}\leq c_3 k,\ \ &{\rm if}\ 2d>n,\vspace{0.2cm}\\
c_2\sum\limits_{m=0}^{M_1}\frac m22^{md}\le c_3M_1\cdot2^{M_1d}\leq
c_4 k, &{\rm if}\ 2d=n.
\end{cases}
$$
Now we give the crucial choice of $k_{j,i}$ for $\triangle_2$. We
put
$$k_{j,i}=[k^{1-\varepsilon}\cdot2^{iz_1}\cdot2^{jz_2}],$$
where $\varepsilon, z_1, z_2$ are positive real numbers such that

$$\delta+\frac{z_2}2<\frac dt,\
0<\frac{z_2-z_1}2<s^\prime-\delta+\frac{2d-n}t\quad {\rm and}\quad
\frac{z_2t}{2d}=\varepsilon\quad {\rm if}\ \delta<\frac dt\ {\rm
and}\ \delta-s^\prime<\frac{2d-n}t,$$ or
$$\delta+s^\prime+\frac{z_1+z_2}2<\frac nt,\
0<\frac{z_1-z_2}2<\delta-s^\prime+\frac{n-2d}t,\
\frac{(z_1+z_2)t}{2n}=\varepsilon\ {\rm if}\, \delta+s^\prime<\frac
nt\ {\rm and}\ \delta-s^\prime>\frac{2d-n}t.$$ Note that the
relation, $0<\varepsilon<1,$ holds obviously. We observe that
\begin{equation*}
\begin{split}
&\sum\limits_{m=M_1+1}^{M_2}\sum\limits_{\stackrel{i+j=m}{i<j}}k_{j,i}\le
c_1k^{1-\varepsilon}\sum\limits_{m=M_1+1}^{M_2}\sum\limits_{\stackrel{i+j=m}{i<j}}
2^{iz_1}\cdot2^{jz_2}
\\
&\le \begin{cases} c_2 k^{1-\varepsilon}\sum\limits_{m=M_1+1}^{M_2}
2^{mz_2}\le c_3 k^{1-\varepsilon}2^{M_2z_2}=c_3k,\  &{\rm if}\ \
\delta<\frac dt\ {\rm
and}\ \delta-s^\prime<\frac{2d-n}t,\\
c_2 k^{1-\varepsilon}\sum\limits_{m=M_1+1}^{M_2}
2^{m\frac{z_1+z_2}2}\le c_3
k^{1-\varepsilon}2^{M_2\frac{z_1+z_2}2}=c_3k, &{\rm if}\ \
\delta+s^\prime<\frac nt\ {\rm and}\ \delta-s^\prime>\frac{2d-n}t,
\end{cases}
\end{split}
\end{equation*}
and, in terms of (\ref{idealPji2}),
\begin{equation*}
\begin{aligned}
\triangle_2 &\leq
c_1\sum\limits_{m=M_1+1}^{M_2}\sum\limits_{\stackrel{i+j=m}{i<j}}
2^{-j\delta-is^\prime}
2^{(in+d(j-i))/t}[k^{1-\varepsilon}\cdot2^{iz_1}\cdot2^{jz_2}]^{-\frac 12}\\
&\le c_2k^{-\frac
12(1-\varepsilon)}\sum\limits_{m=M_1+1}^{M_2}\sum\limits_{\stackrel{i+j=m}{i<j}}
2^{-j(\delta-\frac dt+\frac{z_2}2)}2^{-i(s^\prime-\frac{n-d}t+\frac{z_1}2)}\\
&\le
\begin{cases} c_3 k^{-\frac
12(1-\varepsilon)}\sum\limits_{m=M_1+1}^{M_2} 2^{-m(\delta-\frac
dt+\frac{z_2}2)}\le c_4 2^{-M_2\delta}=c_4k^{-\frac{t\delta}{2d}},\
{\rm if}\ \ \delta<\frac dt\ {\rm and}\
\delta-s^\prime<\frac{2d-n}t,\\
c_3 k^{-\frac 12(1-\varepsilon)}\sum\limits_{m=M_1+1}^{M_2}
2^{-\frac m2(\delta+s^\prime+\frac{z_1+z_2}2-\frac nt)}\le c_4
k^{-\frac {t(\delta+s^\prime)}{2n}},\ {\rm if}\
\delta+s^\prime<\frac nt\ {\rm and}\ \delta-s^\prime>\frac{2d-n}t.
\end{cases}
\end{aligned}
\end{equation*}
 {\tt
Substep 4.2.}\quad  We consider those parts concerning $i\ge j$\,
(i.e., $\triangle_4,\triangle_5$ and $\triangle_6$). We take
\begin{equation*}
M_3= \Big[\frac {\log_2k}{n}\Big]\ \  \ {\rm and}\ \ \ M_4=
\Big[\frac t2 \cdot\frac{\log_2k}{n}\Big].
\end{equation*}
It should be noted that the inequalities, $M_1\neq M_3$ and $M_2\neq
M_4$, hold in general. Then
\begin{equation*}
\begin{split}
\triangle_6&=\sum\limits_{m=M_2+1}^\infty
\sum\limits_{\stackrel{i+j=m}{i\ge j}}\|P_{j,i}\|\le
c_1\sum\limits_{m=M_2+1}^\infty
\sum\limits_{\stackrel{i+j=m}{i\ge j}} 2^{-j\delta}2^{-is^\prime}\\
&\le \begin{cases} c_2 \sum\limits_{m=M_2+1}^\infty 2^{-\frac
m2(\delta+s^\prime)}\le c_3 2^{-M_4\frac {\delta+s^\prime}2}\le c_3
k^{-\frac
{t(\delta+s^\prime)}{4n}},\ \ &{\rm if}\ \ \delta<s^\prime,\\
c_2 \sum\limits_{m=M_2+1}^\infty 2^{-ms^\prime}\le c_3
2^{-M_4s^\prime}\le c_3 k^{-\frac{ts^\prime}{2n}},\ \ &{\rm if}\ \
\delta>s^\prime.
\end{cases}
\end{split}
\end{equation*}
 If $i\ge j$ and $i+j\leq M_3$, we take $k_{j,i}=M_{j,i}+1$ such
that $a_{k_{j,i}}(P_{j,i})=0$ and $\triangle_4=0$. Moreover,
$$
\sum\limits_{m=0}^{M_3}\sum\limits_{\stackrel{i+j=m}{i\ge
j}}k_{j,i}\le
 c_1\sum\limits_{m=0}^{M_3}\sum\limits_{\stackrel{i+j=m}{i\ge j}}2^{ni}
\le c_2\sum\limits_{m=0}^{M_3}2^{mn}\le c_32^{M_3 n}\leq c_4
k.\vspace{0.2cm}
$$
For $\triangle_5$, we put similarly
$$k_{j,i}=[k^{1-\varepsilon}\cdot2^{iz_3}\cdot2^{jz_4}],$$
where $\varepsilon, z_3, z_4$ are positive real numbers such that
$s^\prime+\frac{z_3}2<\frac nt,\
0<\frac{z_3-z_4}2<\delta-s^\prime+\frac nt \ {\rm and} \
\frac{z_3t}{2n}=\varepsilon$. Recall that ${s^\prime}<n/t$. And
observe that
\begin{equation*}
\sum\limits_{m=M_3+1}^{M_4}\sum\limits_{\stackrel{i+j=m}{i\ge
j}}k_{j,i}\le
c_1k^{1-\varepsilon}\sum\limits_{m=M_3+1}^{M_4}\sum\limits_{\stackrel{i+j=m}{i\ge
j}} 2^{iz_3}\cdot2^{jz_4} \le  c_2
k^{1-\varepsilon}\sum\limits_{m=M_3+1}^{M_4} 2^{mz_3}\le c_3
k^{1-\varepsilon}2^{M_4z_3}=c_3k,
\end{equation*}
and, in terms of (\ref{idealPji2}),\vspace{-0.2cm}
\begin{equation*}
\begin{aligned}
\triangle_5 &\leq
c_1\sum\limits_{m=M_3+1}^{M_4}\sum\limits_{\stackrel{i+j=m}{i\ge j}}
2^{-j\delta-is^\prime}
2^{in/t}[k^{1-\varepsilon}\cdot2^{iz_3}\cdot2^{jz_4}]^{-\frac 12}\\
& \le c_2k^{-\frac
12(1-\varepsilon)}\sum\limits_{m=M_3+1}^{M_4}\sum\limits_{\stackrel{i+j=m}{i\ge
j}}
2^{-j(\delta+\frac{z_4}2)}2^{-i(s^\prime-\frac{n}t+\frac{z_3}2)}\\
& \le
 c_3 k^{-\frac
12(1-\varepsilon)}\sum\limits_{m=M_3+1}^{M_4}
2^{-m(s^\prime-\frac{n}t+\frac{z_3}2)}\\
& \le c_4 2^{-M_4s^\prime}=c_4k^{-\frac{ts^\prime}{2n}}.
\end{aligned}
\end{equation*}

Summarizing all the estimates of the six parts (in fact,
$\triangle_2,\, \triangle_3,\,\triangle_5$\, and\ $\triangle_6$) in
each case of (ii)-(iv), we obtain
 the upper bounds in these cases, respectively, as required.
 We wish to mention that, in case (ii), if $\delta>s^\prime$ and
$\delta-s^\prime<\frac{2d-n}t$, then the inequality $2d>n$ is valid,
and similarly in case (iv), the relation $2d<n$ holds.
 \\
 {\tt Step 5}.\quad The lower estimate of $a_k({\rm id}, B_1, B_2)$.
 Consider the following diagram
\begin{equation}
\begin{CD}
\ell_{p_1}^{M_{j,i}} @>S_{j,i}>>
\ell_{q_1}(2^{j\delta}\ell_{p_1}(w))
\\
@VV{\rm id_1}V @VV{\rm id}V\\
\ell_{p_2}^{M_{j,i}} @<T_{j,i}<< \ell_{q_2}(\ell_{p_2})
\end{CD}
\end{equation}
Here,
\begin{equation*}
\begin{array}{rl}
(S_{j,i}\eta)_{u,v}:=&\left\{
\begin{array}{ll}\eta_{\varphi(u,v)}\,
&{\rm if}\quad (u,v)\in I_{j,i}, \\ 0 \, &{\rm otherwise,}
\end{array}\right.\vspace{0.1cm}\\
(T_{j,i}\lambda)_{\varphi(u,v)}:=&\lambda_{u,v},\quad\quad\quad
(u,v)\in I_{j,i},\end{array}
\end{equation*}
and $\varphi$ denotes a bijection of $I_{j,i}$ onto $\{1, \ldots,
M_{j,i}\},\, j\in\mathbb{N}_0,\, i\in\mathbb{N}_0$; cf. (\ref{Ij0})
and (\ref{Iji}). Observe that
\begin{equation*}
\begin{array}{ll}
S_{j,i}\in \mathcal {L}\left(\ell_{p_1}^{M_{j,i}},\,
\ell_{q_1}(2^{j\delta}\ell_{p_1}(\alpha))\right) \quad  &{\rm
and}\quad \|S_{j,i}\|=2^{j\delta+is^\prime},\vspace{0.1cm}\\
T_{j,i}\in  \mathcal {L}\left(\ell_{q_2}(\ell_{p_2}),\,
\ell_{p_2}^{M_{j,i}}\right) \quad  &{\rm and}\quad \|T_{j,i}\|=1.
\end{array}
\end{equation*}
Hence we obtain
\begin{equation}\label{diagST}
a_k(\,{\rm id}_1)\leq\|S_{j,i}\|\,\|T_{j,i}\|\,a_k(\,{\rm id}).
\end{equation}

In the following four points, we first assume $p_2\neq\infty$ if
$0<p_1\le1$.

 (a)\, Let $\frac 1t<\frac\delta d\le \frac{s^\prime}n$.  We
consider the case $i=0,\,j\ge\frac 2d$. Lemma \ref{Nji} implies that
$N_{j,0}\sim 2^{d j}$. In view of (\ref{rankpiecelow}), we observe
that either $M_{j,1}$ or $M_{j,2}$ is no smaller than $N_{j,0}/2$.
So by (\ref{rankpiece}), we may assume that
$N:=M_{j,1}=|I_{j,1}|\sim 2^{jd}$. Moreover,
$$\|S_{j,1}\|\leq C2^{j\delta}\quad{\rm and}\quad\|T_{j,1}\|=1.$$
Put $m=\left[\frac N4\right]\sim2^{jd-2}.$ And for sufficiently
large $N$ we have $m\ge N^{2/t}$ since $t>2.$ Consequently, we
observe by Lemma \ref{an1} that\vspace{-0.2cm}
\begin{equation*}
a_m({\rm id}_1,\,\ell_{p_1}^N,\,\ell_{p_2}^N)\sim N^{\frac
1t}m^{-\frac 12}\sim2^{(jd-2)(\frac 1t-\frac 12)}.
\end{equation*}
Using (\ref{diagST}), we obtain
\begin{equation*}
a_{2^{jd-2}}({\rm id})\ge C_12^{-j\delta}2^{(jd-2)(\frac 1t-\frac
12)}\ge C_22^{(jd-2)(\frac 1t-\frac 12-\frac\delta d)}.
\end{equation*}
Then the monotonicity of the approximation numbers implies that for
any $ k\in \mathbb{N}$
\begin{equation}
a_k({\rm id})\ge C_3k^{-(\frac\delta d+\frac 12-\frac 1t)}.
\end{equation}

(b)\, Let $\frac 1t<\frac{s^\prime}n\le \frac\delta d$. We consider
the case $j=0,\,j\ge\frac 2d$. Then, in a similar manner as above,
we may assume that $N:=M_{0,i+1}=|I_{0,i+1}|\sim 2^{in}$. Moreover,
$$\|S_{0,i+1}\|\leq C2^{is^\prime}\quad{\rm and}\quad\|T_{0,i+1}\|=1.$$
Also put $m=\left[\frac N4\right]\sim2^{ni-2}.$ Hence we have
similarly for any $ k\in \mathbb{N}$
\begin{equation}
a_k({\rm id})\ge Ck^{-(\frac{s^\prime}n+\frac 12-\frac 1t)}.
\end{equation}

(c)\, Let $\frac\delta d\le\frac 1t$ and $\frac\delta d\le
\frac{s^\prime}n$. We select the same $N,\,S$, and\ $T$ as in point
(a) and take $m=\left[N^{\frac 2t}\right]\leq\frac N4$ for
sufficiently large $N.$ Then $N^{\frac 1t}m^{-\frac 12}\sim 1.$
Hence by Lemma \ref{an1} and (\ref{diagST}) we obtain
\begin{equation*}
a_m({\rm id})\ge C2^{-j\delta}=C2^{-jd\frac 2t\frac{t\delta}{2d}},
\end{equation*}
and then for any $ k\in \mathbb{N}$
\begin{equation}\label{delts}
a_k({\rm id})\ge Ck^{-\frac{t\delta}{2d}}.
\end{equation}

(d)\, Let $\frac{s^\prime}n\le\frac 1t$ and $\frac{s^\prime}n\le
\frac\delta d$. We select the same $N,\,S$, and\ $T$ as in point (b)
and take $m=\left[N^{\frac 2t}\right]$ in the same way as in point
(c). Then analogously
\begin{equation*}
a_m({\rm id})\ge C2^{-is^\prime}=C2^{-ni\frac
2t\frac{ts^\prime}{2n}},
\end{equation*}
and in consequence, for any $k\in \mathbb{N}$
\begin{equation}
a_k({\rm id})\ge Ck^{-\frac{ts^\prime}{2n}}.
\end{equation}

If $0<p_1\le1$ and $p_2=\infty$, we have $t=\infty$ and consider two
cases, $0<\frac\delta d\le\frac {s^\prime}n$\
 or $0<\frac {s^\prime}n<
\frac\delta d$. And we choose $m=\big[\frac N2\big]$ where $N$ is
taken in the same way as in point (a) or (b), respectively, now
using (\ref{an1inflow}) instead of Lemma \ref{an1}.

Finally, we mention that in case (ii), the condition
$\delta>s^\prime$ implies the inequality
$\frac{s^\prime}n<\frac\delta d$.

 The proof of the proposition is now complete.
\end{proof}

\begin{re}
In the situation considered in Proposition \ref{an4}, how do the
approximation numbers behave, if ${s^\prime}<n/t$ and,
$\delta=s^\prime$ or $\delta-s^\prime=\frac{2d-n}t$? The lower bound
given in (\ref{delts}) may be the exact asymptotic estimate in some
special cases.
\end{re}

\begin{prop}\label{an5}
Suppose $0< p_1 \le p_2 \le 2$ or $2\le p_1 \le p_2 \le \infty$.
Then
 \vspace{-0.2cm}
\begin{equation}
a_{k}\Big({\rm id}, \ell_{q_1}(2^{j\delta}\ell_{p_1}(w)),
\ell_{q_2}(\ell_{p_2})\Big) \sim  k^{-\varkappa},
\end{equation}
with
\begin{equation}
\varkappa=\min\big(\frac\delta d,\frac {s^\prime}n\big).
\end{equation}
\end{prop}

\begin{proof} The upper bound can be proved in the same way
as in the proof of Theorem 6 in \cite{LS11} using $1/s$ instead of
$1/r-1/p$, where the latter notations follow from \cite{LS11}. In
fact, in the estimate from above,
 in view of Lemma \ref{an1} (i), we
obtain that for any $s>0$, \vspace{-0.2cm}
\begin{equation}
 L_{s,\infty}^{(a)}({\rm id}, \ell_{p_1}^{M_{j,i}},
 \ell_{p_2}^{M_{j,i}})\leq
C
\begin{cases}
2^{(in+d(j-i))/s},\ \ \  & 0\le i<j,\\
2^{in/s}, & 0\le j\le i.
\end{cases}
\end{equation}

In the estimate from below we only need to consider two cases. If
$0<\frac\delta d\le\frac {s^\prime}n$\ we take the same $N,\ S$\
and\ $T$\ as in point (a) of Step 5 of the last proof. If\ $0<\frac
{s^\prime}n< \frac\delta d$\ we choose the same $N,\ S$\ and\ $T$\
as in point (b) therein.
\end{proof}
\begin{prop}\label{an6}
Suppose $\tilde{p}<p_2<p_1\leq\infty$ where $\frac
1{\tilde{p}}=\min\big(\frac\delta d,\frac {s^\prime}n\big)+\frac
1{p_1}$. Then
 \vspace{-0.2cm}
\begin{equation}
a_{k}\Big({\rm id}, \ell_{q_1}(2^{j\delta}\ell_{p_1}(w)),
\ell_{q_2}(\ell_{p_2})\Big) \sim  k^{-\varkappa},
\end{equation}
with
\begin{equation}
\varkappa=\min\big(\frac\delta d,\frac {s^\prime}n\big)+\frac
1{p_1}-\frac 1{p_2}.
\end{equation}
\end{prop}

\begin{proof}
The proof of the upper bound may be again finished as in the proof
of Theorem 6 in \cite{LS11} with $1/r-1/p$ replaced by
$1/s-1/{p_1}+1/{p_2}$, where the former notations follow from
\cite{LS11}. Indeed, in terms of Lemma \ref{an2}, we observe that
for any $s>0$, \vspace{-0.2cm}
\begin{equation}
 L_{s,\infty}^{(a)}({\rm id}, \ell_{p_1}^{M_{j,i}},
 \ell_{p_2}^{M_{j,i}})\leq
C
\begin{cases}
2^{(in+d(j-i))(\frac 1s-\frac 1{p_1}+\frac 1{p_2})},\ \ \  & 0\le i<j,\\
2^{in(\frac 1s-\frac 1{p_1}+\frac 1{p_2})}, & 0\le j\le i.
\end{cases}
\end{equation}
In the estimate from below, once more we follow the proof of Step 5
of Proposition \ref{an4}. In order to guarantee the compactness of
the embeddings, here we only need to consider two cases,
 $\frac 1{p^\ast}<\frac\delta d\le\frac {s^\prime}n$\
 or $\frac 1{p^\ast}<\frac {s^\prime}n<
\frac\delta d$. And we choose $m=\big[\frac N2\big]$ where $N$ is
taken in the same way as in point (a) or (b), respectively, now
using Lemma \ref{an2} instead of Lemma \ref{an1}.
 \end{proof}
\subsection{Kolmogorov numbers of sequence spaces}\label{kns}

Now, we turn our attention to Kolmogorov numbers. To begin with, we
shall collect some information on estimates for the Euclidean ball.
Lemma \ref{kn1} follows trivially from results of Gluskin
\cite{Gl83} and Edmunds and Triebel \cite{ET96}.

\begin{lemma}\label{kn1} Let $N\in\mathbb{N}$.\vspace{-0.2cm}
\begin{enumerate}
\item[{\rm (i)}]\ If $1\le p_1\le p_2\le 2$ and $k\le\frac{N}{4}$, then
$$d_{k}\left({\rm id}, \ell_{p_1}^N, \ell_{p_2}^N\right)\sim 1.$$
\vspace{-0.6cm}

\item[{\rm (ii)}]\ If $1\le p_1 < 2 < p_2< \infty$ and $k\le\frac{N}{4}$,\,
then
$$d_{k}\left({\rm id}, \ell_{p_1}^N, \ell_{p_2}^N\right)\sim \min\{1,N^{\frac 1{p_2}}k^{-\frac 1{2}}\}.$$
\vspace{-0.6cm}

\item[{\rm (iii)}]\ If $2<p_1 = p_2\le \infty$ and $k\le N$, then
$$d_k\left({\rm id}, \ell_{p_1}^N, \ell_{p_2}^N\right)\sim 1.$$
\vspace{-0.6cm}

\item[{\rm (ii)}]\ If $2 \le p_1 < p_2 < \infty$\,  and $k\le N$, then
$$d_{k}\left({\rm id}, \ell_{p_1}^N, \ell_{p_2}^N\right)
\sim \xi^\theta,$$ where $\xi=\min\{1,N^{\frac 1{p_2}}k^{-\frac
1{2}}\},\,\theta = \frac{1/{p_1}-1/{p_2}}{1/2-1/{p_2}}.$
\end{enumerate}
\end{lemma}

The following lemma are a simply corollary of Lemma \ref{kn1}. And
the proof can be finished in the same manner as in the proof of
Lemma 10 in Skrzypczak \cite{Sk05}.
\begin{lemma}\label{kn3}  Suppose $1\le p_1 < 2 < p_2< \infty$. Then
there is a positive constant $C$ independent of $N$ and $k$ such
that
\begin{equation}\label{klp2p}
d_{k}\left({\rm id},\ell_{p_1}^N,\ell_{p_2}^N\right)\leq C\, \left\{
\begin{array}{ll}
1 & {\rm if}\, k\leq N^{\frac 2{p_2}},\\
N^{\frac 1{p_2}}k^{-\frac 1{2}}\quad & {\rm if}\, N^{\frac
2{p_2}}<k\leq
N,\\
0 & {\rm if}\, k>N.
\end{array}\right.
\end{equation}
\end{lemma}

Now, we go on to make preparations for the estimates of Kolmogorov
numbers of related embeddings for the quasi-Banach case with\
$0<p_1< 1$\ or\ $0<p_2< 1,$\ and for several cases left over when\
$p_2=\infty$. The following result, Lemma \ref{knify}, is due to
Kashin \cite{Ka77}, Garnaev and Gluskin \cite{GG84} and Vyb\'iral
\cite{Vy08}, cf. also \cite{LGM96}.

\beg{lemma}\label{knify}
Let $N\in\mathbb{N}$.\vspace{-0.2cm}
\beg{enumerate}
\item[$(i)$]\ If~ $1\le p < 2$ and $k\le\frac{N}{4}$,\ then
$$k^{-1/2}\preceq  d_{k}\left({\rm id}, \ell_{p}^N, \ell_{\infty}^N\right)\preceq
 k^{-1/2}\Big(\log\big(\frac{eN}{k}\big)\Big)^{3/2}.$$\vspace{-0.6cm}
\item[$(ii)$]\ If~ $2 \le p < \infty$\, and $ k\le
N$, then
$$\frac 14\min\Big\{1,\Big(c_1\frac{\log(1+\frac N{k-1})}{k-1}\Big)^{1/p}\Big\}\le
 d_{k}\left({\rm id}, \ell_{p}^N, \ell_{\infty}^N\right)\le
\min\Big\{1,\Big(c_2\frac{\log(1+\frac
N{k-1})}{k-1}\Big)^{1/p}\Big\}$$ are valid for certain absolute
constants $c_1>0$\ and\  $c_2>0$.
\item[$(iii)$]\ If~ $0 < p_1<1,\ p_1 <p_2 \le \infty$\, and\ $ k\le
N$, then
$$d_{k}\big({\rm id}, \ell_{p_1}^N,
\ell_{p_2}^N\big)=d_{k}\big({\rm id}, \ell_{\min(1,p_2)}^N,
\ell_{p_2}^N\big).$$
\end{enumerate}
\end{lemma}
The following lemma is a simple corollary of Lemma \ref{knify}. And
the proof is similar to that of Lemma 10 in \cite{Sk05}.
\beg{lemma}\label{knifyupp12}
 \ Let $1\le p_1 < 2$\
and\ $N=1, 2, 3, \ldots.$\ Then there is a positive constant $C$
independent of $N$ and $k$ such that \beq \label{kn12}
 d_{k}\left({\rm id}, \ell_{p_1}^N,
\ell_{\infty}^N\right)\le C\ \beg{cases}
 k^{-1/2}\Big(\log\big(\frac{4eN}{k}\big)\Big)^{3/2}\, ~ &{\rm if}\ 0<k\le
 N,\\0 &{\rm if}\ k>N.
\end{cases}
\enq
\end{lemma}

The following estimate is due to Vyb\'iral \cite{Vy08}.
 \beg{lemma}\label{kn021ify}
If~ $0<p_2\le p_1\le\infty,$ then there is a constant $c,\ 0<c\le
1,$\ such that
$$
d_{[ck]+1}\big({\rm id}, \ell_{p_1}^{2k},
\ell_{p_2}^{2k}\big)\succeq n^{1/{p_2}-1/{p_1}},\quad k\in
\mathbb{N}.
$$
\end{lemma}

Now we are ready to deal with the Kolmogorov numbers of embeddings
of related sequence spaces in the quasi-Banach setting, $0< p, q\le
\infty$.

\begin{prop}\label{kn4}
Suppose $0< p_1 < 2 < p_2 \le \infty$.
 \beg{enumerate}
\item[$(i)$] If\, $\min\big(\frac\delta d,\frac
{s^\prime}n\big)>\frac 1{p_2}$, then
 \vspace{-0.2cm}
\begin{equation}
d_{k}\Big({\rm id}, \ell_{q_1}(2^{j\delta}\ell_{p_1}(w)),
\ell_{q_2}(\ell_{p_2})\Big) \sim k^{-\min\big(\frac\delta d,\frac
{s^\prime}n\big)
 +\frac 1{p_2}-\frac 12}.
\end{equation}
 \vspace{-0.4cm}
\item[$(ii)$] If $\delta > s^\prime$ and either
 $
\begin{cases}
\delta<\frac d{p_2},\\ \delta-s^\prime<\frac{2d-n}{p_2},
\end{cases}\ \ {\rm or}\ \
\begin{cases}
\delta+s^\prime<\frac n{p_2},\\ \delta-s^\prime>\frac{2d-n}{p_2},
\end{cases} $
then
\begin{equation}
d_{k}\Big({\rm id}, \ell_{q_1}(2^{j\delta}\ell_{p_1}(w)),
\ell_{q_2}(\ell_{p_2})\Big) \sim k^{-\frac
{s^\prime}n\cdot\frac{p_2}2}.
\end{equation}
 \vspace{-0.4cm}
\item[$(iii)$] If\ $\delta<\frac d{p_2},\ \delta<s^\prime<\frac n{p_2}$
and $\delta-s^\prime<\frac{2d-n}{p_2}$, then \vspace{-0.2cm}
\begin{equation}k^{-\frac
{p_2}2\cdot\min(\frac\delta d,\frac {s^\prime}n)}\preceq
d_{k}\Big({\rm id}, \ell_{q_1}(2^{j\delta}\ell_{p_1}(w)),
\ell_{q_2}(\ell_{p_2})\Big)\preceq k^{-\frac {p_2}2
\cdot\min(\frac\delta d,\frac {\delta+s^\prime}{2n})}.
\end{equation}
 \vspace{-0.4cm}
\item[$(iv)$] If\ $\delta<s^\prime,\ \delta+ s^\prime<\frac n{p_2}$ and
$\delta-s^\prime>\frac{2d-n}{p_2}$, then \vspace{-0.2cm}
\begin{equation}k^{-\frac
{p_2}2\cdot\min(\frac\delta d,\frac {s^\prime}n)}\preceq
d_k\Big({\rm id}, \ell_{q_1}(2^{j\delta}\ell_{p_1}(w)),
\ell_{q_2}(\ell_{p_2})\Big)\preceq k^{-\frac
{p_2}n\cdot\min(\delta,\frac {\delta+s^\prime}4)}.
\end{equation}
 \end{enumerate}
\end{prop}

\begin{proof}
We consider three cases according to the distributions of $p_1$\
and\ $p_2$. Note that points (ii)-(iv) vanish when $p_2=\infty$.

Case 1.\ Assume $1\le  p_1 < 2 < p_2 < \infty$. The proof of the
proposition can be finished in the same manner as in the proof of
Proposition \ref{an4}, with Lemma \ref{an1} and (\ref{alp2p})
replaced by Lemma \ref{kn1} and (\ref{klp2p}), respectively. The
only change is that $t=\min(p_1^\prime,p_2)$ is replaced by $p_2$ in
this case.

Case 2.\ Assume $1\le  p_1 < 2$ and $p_2=\infty$. We proceed as
above, now using Lemma \ref{knify} (i) and (\ref{kn12}) instead of
Lemma \ref{kn1} and (\ref{klp2p}), respectively. Related
computations of ideal quasi-norms herein are similar to the
counterpart of entropy numbers, cf. \cite{ET96,KLSS06b}. Indeed,
$$
L_{s,\infty}^{(d)}({\rm id}, \ell_{p_1}^N, \ell_\infty^N)\leq C
N^{\frac 1s-\frac 12},\ \ \ \frac 1s>\frac 12,\ \ N\in\mathbb{N}.
$$
 Moreover,
in the estimate of lower bounds, because of $p_2=\infty$, we only
need to consider two cases, $0<\frac\delta d\le\frac {s^\prime}n$\
 or $0<\frac {s^\prime}n<
\frac\delta d$, in the same way as in point (a) or (b) of Step 5 in
the proof of Proposition \ref{an4}, respectively, and take
$m=\left[\frac N4\right]$ in both cases based on Lemma \ref{knify}
(i).

Case 3.\ Assume $0<p_1<1$ and $2<p_2\le\infty$. We first transform
the problem of this case to the above two cases (i.e., Case 1 for
$p_2<\infty$, Case 2 for $p_2=\infty$), by virtue of Lemma
\ref{knify} (iii), and follow trivially them respectively. Note that
the exact upper estimate here may also be provided by the
corresponding statement about approximation numbers, cf. Proposition
\ref{an4} and (\ref{acd}).

\end{proof}

\begin{prop}\label{kn5}
Suppose $ 2\le p_1 < p_2 \le \infty$. We set $\theta =
\frac{1/{p_1}-1/{p_2}}{1/2-1/{p_2}}$.
 \beg{enumerate}
\item[$(i)$] If\, $\min\big(\frac\delta d,\frac
{s^\prime}n\big)>\frac \theta{p_2}$, then
 \vspace{-0.2cm}
\begin{equation}
d_{k}\Big({\rm id}, \ell_{q_1}(2^{j\delta}\ell_{p_1}(w)),
\ell_{q_2}(\ell_{p_2})\Big) \sim k^{-\min\big(\frac\delta d,\frac
{s^\prime}n\big)
 -\frac 1{p_1}+\frac 1{p_2}}.
\end{equation}
 \vspace{-0.4cm}
\item[$(ii)$] If $\delta > s^\prime$ and either
 $
\begin{cases}
\delta<\frac d{p_2}\theta,\\ \delta-s^\prime<\frac{2d-n}{p_2}\theta,
\end{cases}\ \ {\rm or}\ \
\begin{cases}
\delta+s^\prime<\frac n{p_2}\theta,\\
\delta-s^\prime>\frac{2d-n}{p_2}\theta,
\end{cases} $
then
\begin{equation}
d_{k}\Big({\rm id}, \ell_{q_1}(2^{j\delta}\ell_{p_1}(w)),
\ell_{q_2}(\ell_{p_2})\Big) \sim k^{-\frac
{s^\prime}n\cdot\frac{p_2}2}.
\end{equation}
 \vspace{-0.4cm}
\item[$(iii)$] If\ $\delta<\frac d{p_2}\theta,\ \delta<s^\prime<\frac n{p_2}\theta$
and $\delta-s^\prime<\frac{2d-n}{p_2}\theta$, then \vspace{-0.2cm}
\begin{equation}k^{-\frac
{p_2}2\cdot\min(\frac\delta d,\frac {s^\prime}n)}\preceq
d_{k}\Big({\rm id}, \ell_{q_1}(2^{j\delta}\ell_{p_1}(w)),
\ell_{q_2}(\ell_{p_2})\Big)\preceq k^{-\frac {p_2}2
\cdot\min(\frac\delta d,\frac {\delta+s^\prime}{2n})}.
\end{equation}
 \vspace{-0.4cm}
\item[$(iv)$] If\ $\delta<s^\prime,\ \delta+ s^\prime<\frac n{p_2}\theta$ and
$\delta-s^\prime>\frac{2d-n}{p_2}\theta$, then \vspace{-0.2cm}
\begin{equation}k^{-\frac
{p_2}2\cdot\min(\frac\delta d,\frac {s^\prime}n)}\preceq
d_k\Big({\rm id}, \ell_{q_1}(2^{j\delta}\ell_{p_1}(w)),
\ell_{q_2}(\ell_{p_2})\Big)\preceq k^{-\frac
{p_2}n\cdot\min(\delta,\frac {\delta+s^\prime}4)}.
\end{equation}
 \end{enumerate}
\end{prop}

\begin{proof}
We consider two cases as follows. Note that points (ii)-(iv) vanish
when $p_2=\infty$.

Case 1.\ Assume  $2\le  p_1 < p_2 < \infty$. We only sketch the
proof since, once more, we can use the similar reasoning. To shorten
notations we shall put $\tau=\frac {p_2}\theta,\,h=\frac 2\theta$
and $\frac 1s=\frac 1\gamma+\frac 1h$\, for any $s>0$. These simple
transformations lead us to follow trivially from the proof of
Proposition \ref{an4}. Please note that in the upper estimate, by
Lemma \ref{kn1},
\begin{equation}\label{idealseh}
 L_{h,\infty}^{(d)}({\rm id}, \ell_{p_1}^{M_{j,i}},
 \ell_{p_h}^{M_{j,i}})\leq
C
\begin{cases}
2^{(in+d(j-i))/\tau},\ \ \  & 0\le i<j,\\
2^{\frac {in}\tau}, & 0\le j\le i,
\end{cases}\indent\indent\indent\
\end{equation}
\begin{equation}\label{idealsh}
L_{s,\infty}^{(d)}({\rm id}, \ell_{p_1}^{M_{j,i}},
\ell_{p_2}^{M_{j,i}})\leq C
\begin{cases}
2^{(in+d(j-i))(\frac 1\tau+\frac 1\gamma)},\ \ \ \ & 0\le i<j,\ \frac 1s>\frac 1h,\\
2^{in(\frac 1\tau+\frac 1\gamma)},& 0\le j\le i,\ \frac 1s>\frac 1h.
\end{cases}
\end{equation}
Similarly, with respect to the estimate from below,
$t=\min(p_1^\prime,p_2)$ is replaced by $\tau=\frac {p_2}\theta$ in
related places. One can consult our previous paper \cite{ZF10} for
further details.

Case 2.\ Assume $2\le p_1< p_2=\infty$. We proceed as above, now
using Lemma \ref{knify} (ii) instead. Again, computations of
corresponding operator ideal quasi-norms start up as below,
$$
L_{s,\infty}^{(d)}({\rm id}, \ell_{p_1}^N, \ell_\infty^N)\leq C
N^{\frac 1s-\frac 1{p_1}},\ \ \ \frac 1s>\frac 1{p_1},\ \
N\in\mathbb{N}.
$$
 In the estimate of lower bounds,
 we only need to consider two cases,
$0<\frac\delta d\le\frac {s^\prime}n$\
 or $0<\frac {s^\prime}n<
\frac\delta d$, as in Case 2 of the last proof, and take $m=N$ in
both cases.
\end{proof}

\begin{prop}\label{kn6}
Suppose $0< p_1 \le p_2 \le 2$ or $2< p_1 = p_2 \le \infty$. Then
 \vspace{-0.2cm}
\begin{equation}
d_{k}\Big({\rm id}, \ell_{q_1}(2^{j\delta}\ell_{p_1}(w)),
\ell_{q_2}(\ell_{p_2})\Big) \sim  k^{-\varkappa},
\end{equation}
with
\begin{equation}
\varkappa=\min\big(\frac\delta d,\frac {s^\prime}n\big).
\end{equation}
\end{prop}

\begin{proof} The upper estimate is provided by the
corresponding statement about approximation numbers, cf. Proposition
\ref{an5} and (\ref{acd}).

In the lower estimate, once more we follow the proof of Step 5 of
Proposition \ref{an4}.  If $0<\frac\delta d\le\frac {s^\prime}n$\ we
take the same $N,\ S$\ and\ $T$\ as in point (a). If\ $0<\frac
{s^\prime}n< \frac\delta d$\ we take $N, S$\ and\ $T$\ the same as
in point (b). Moreover, in each of these two cases we choose
$m=[\frac N4]$\ (if $p_2\ge 1$) or\ $m=\big[\frac c2\cdot N\big]$\
(if $p_1\le p_2< 1$) where $c$ is the constant from Lemma
\ref{kn021ify}, and we use Lemma \ref{kn1} (if $p_1\ge 1$) or, Lemma
\ref{knify} (iii) and Lemma \ref{kn021ify} (if $p_1\le p_2< 1$) or,
Lemma \ref{kn1} and Lemma \ref{knify} (iii) (if $p_1< 1\le p_2$),
instead of Lemma \ref{an1}.
\end{proof}

\begin{prop}\label{kn7}
Suppose $0<\tilde{p}<p_2<p_1\leq\infty$ where $\frac
1{\tilde{p}}=\min\big(\frac\delta d,\frac {s^\prime}n\big)+\frac
1{p_1}$. Then
 \vspace{-0.2cm}
\begin{equation}
d_{k}\Big({\rm id}, \ell_{q_1}(2^{j\delta}\ell_{p_1}(w)),
\ell_{q_2}(\ell_{p_2})\Big) \sim  k^{-\varkappa},
\end{equation}
with
\begin{equation}
\varkappa=\min\big(\frac\delta d,\frac {s^\prime}n\big)+\frac
1{p_1}-\frac 1{p_2}.
\end{equation}
\end{prop}

\begin{proof}
Once again the estimate from above is provided by the corresponding
statement about approximation numbers, cf. Proposition \ref{an6} and
(\ref{acd}).

Again, in the estimate from below we follow the proof of Step 5 of
Proposition \ref{an4}. We only need to consider two cases, $\frac
1{p^\ast}<\frac\delta d\le\frac {s^\prime}n$\ or $\frac
1{p^\ast}<\frac {s^\prime}n< \frac\delta d$, as in the proof of
Proposition \ref{an6}. And in each case we choose $m=\big[\frac
c2\cdot N\big]$, where $c$ is the constant from Lemma
\ref{kn021ify}.
\end{proof}

\subsection{Gelfand numbers of sequence spaces}\label{gns}

In this subsection we deal with  Gelfand numbers. First, we collect
some necessary information on the behaviour of $c_k({\rm id},
\ell_{p_1}^N, \ell_{p_2}^N)$, cf. \cite{FPRU,Gl83, Pie78, Vy08},
(\ref{dualc*d}) and (\ref{duald*c}).
\begin{lemma}\label{gn1} \ Let $N\in\mathbb{N}$.\vspace{-0.2cm}
\beg{enumerate}
\item[$(i)$] If $2\le p_1\le
p_2\le\infty$ and $k\le\frac{N}{4}$\, then \vspace{-0.2cm}
$$c_k\left({\rm id}, \ell_{p_1}^N, \ell_{p_2}^N\right)\sim 1.$$
\vspace{-0.8cm}
\item[$(ii)$] If $1< p_1 < 2 < p_2\le \infty$ and $k\le\frac{N}{4}$\,
then \vspace{-0.2cm}
$$c_{k}\left({\rm id}, \ell_{p_1}^N, \ell_{p_2}^N\right)\sim
\min\{1,N^{\frac 1{p_1^\prime}}k^{-\frac 1{2}}\}.$$ \vspace{-0.8cm}
\item[$(iii)$] If $1\le p_1 = p_2< 2$ and $k\le N$,\, then
$$c_k\left({\rm id}, \ell_{p_1}^N, \ell_{p_2}^N\right)\sim 1.$$
\vspace{-0.6cm}
\item[$(iv)$] If $1 < p_1 < p_2 \le 2$\, and $k\le N$, then
$$c_k\left({\rm id}, \ell_{p_1}^N, \ell_{p_2}^N\right)
\sim \xi^{\theta_1},$$\vspace{-0.2cm}
 where $\xi=\min\{1,N^{\frac
1{p_1^\prime}}k^{-\frac 1{2}}\},\,\theta_1 =
\frac{1/{p_2^\prime}-1/{p_1^\prime}}{1/2-1/{p_1^\prime}}.$
\end{enumerate}
\end{lemma}
\

The proof of this lemma follows by (\ref{dualc*d}), (\ref{duald*c})
and Lemma \ref{kn1}.

The following result is due to Foucart et al.\cite{FPRU}. Note that
the definition of Gelfand widths is used in \cite{FPRU}. Here we
refer to it in our words.

 \beg{lemma}\label{gnupp}
\ Let $1\le k\le N<\infty$. \beg{enumerate}
\item[$(i)$]\ If~ $0< p_1\le 1$ and $2 < p_2\le\infty$\ then
there exist constants $C_1,\ C_2>0$\ depending only on\ $p_1$ and\
$p_2$\ such that
$$
C_1\min\Big\{1,\frac{\ln\big(\frac
 N{k-1}\big)+1}{k-1}\Big\}^{1/{p_1}-1/{p_2}} \le
 c_k\left({\rm id}, \ell_{p_1}^N, \ell_{p_2}^N\right)\le
C_2
 \min\Big\{1,\frac{\ln\big(\frac
 N{k-1}\big)+1}{k-1}\Big\}^{1/{p_1}-1/2}.$$

\item[$(ii)$]\ If~ $0< p_1\le 1$ and $p_1 < p_2\le 2$\ then
there exist constants $C_1,\ C_2>0$\ depending only on\ $p_1$ and\
$p_2$\ such that
$$
C_1\min\Big\{1,\frac{\ln\big(\frac
 N{k-1}\big)+1}{k-1}\Big\}^{1/{p_1}-1/{p_2}} \le c_k\left({\rm id},
\ell_{p_1}^N, \ell_{p_2}^N\right)\le
C_2\min\Big\{1,\frac{\ln\big(\frac
 N{k-1}\big)+1}{k-1}\Big\}^{1/{p_1}-1/{p_2}}.$$

\end{enumerate}
\end{lemma}

\begin{re}For the upper bounds, there is another result given by
Vyb\'iral \cite{Vy08}, cf. Lemma 4.11, with a slight difference
between them on the log-factors. But they are equivalent for our
estimates of related upper bounds considered in Theorem \ref{T3}.
\end{re}

 \beg{lemma}\label{gnlow}
\ Let $ k\in \mathbb{N}$. \beg{enumerate}
\item[$(i)$]\ If~ $0< p_1\le 1$ and $2 \le p_2\le\infty$\ then
\beq\label{gnlow2}
 c_{k}\left({\rm id}, \ell_{p_1}^{2k},
\ell_{p_2}^{2k}\right)\succeq
 k^{1/2-1/{p_1}}.
 \enq

\item[$(ii)$]\ If~ $0< p_1\le 1$ and $p_1 < p_2\le 2$\ then
\beq\label{gnlow0} c_k\left({\rm id}, \ell_{p_1}^{2k},
\ell_{p_2}^{2k}\right)\succeq
 k^{1/{p_2}-1/{p_1}}.
 \enq

\end{enumerate}
\end{lemma}
The proof of this lemma follows literally \cite{Vy08}, p.~567, by
the multiplicativity of Gelfand numbers. In fact, The point (ii) of
Lemma \ref{gnupp} may also imply point (ii) of Lemma \ref{gnlow}.

 \beg{lemma}\label{gn021ify}
If~ $1\le k\le N<\infty$\ and $0<p_2\le p_1\le\infty$, then
$$
c_k\big({\rm id}, \ell_{p_1}^N, \ell_{p_2}^N\big)=
(N-k+1)^{1/{p_2}-1/{p_1}}.
$$
\end{lemma}
The proof of this lemma follows literally \cite{Pie78}, Section
11.11.4, see also \cite{Pin85}. Indeed the original proof is used
only for the Banach space case $1\le p_2\le p_1\le\infty$. However,
the same proof works also in the quasi-Banach case\ $0<p_2\le
p_1\le\infty$.

Now we show some asymptotic estimates of Gelfand numbers of
embeddings between related sequence spaces in the quasi-Banach
setting, $0< p, q\le \infty$.

\begin{prop}\label{gn4}
Suppose $0< p_1 < 2 < p_2 \le \infty$.

 \beg{enumerate}
\item[$(i)$] If\, $\min\big(\frac\delta d,\frac
{s^\prime}n\big)>\frac 1{p_1^\prime}$, then
 \vspace{-0.2cm}
\begin{equation}
c_{k}\Big({\rm id}, \ell_{q_1}(2^{j\delta}\ell_{p_1}(w)),
\ell_{q_2}(\ell_{p_2})\Big) \sim k^{-\min\big(\frac\delta d,\frac
{s^\prime}n\big)
 -\frac 1{p_1}+\frac 12}.
\end{equation}
 \vspace{-0.4cm}
\item[$(ii)$] If $\delta > s^\prime$ and either
 $
\begin{cases}
\delta<\frac d{p_1^\prime},\\
\delta-s^\prime<\frac{2d-n}{p_1^\prime},
\end{cases}\ \ {\rm or}\ \
\begin{cases}
\delta+s^\prime<\frac n{p_1^\prime},\\
\delta-s^\prime>\frac{2d-n}{p_1^\prime},
\end{cases} $
then
\begin{equation}
c_{k}\Big({\rm id}, \ell_{q_1}(2^{j\delta}\ell_{p_1}(w)),
\ell_{q_2}(\ell_{p_2})\Big) \sim k^{-\frac
{s^\prime}n\cdot\frac{p_1^\prime}2}.
\end{equation}
 \vspace{-0.4cm}
\item[$(iii)$] If\ $\delta<\frac d{p_1^\prime},\ \delta<s^\prime<\frac n{p_1^\prime}$
and $\delta-s^\prime<\frac{2d-n}{p_1^\prime}$, then \vspace{-0.2cm}
\begin{equation}k^{-\frac
{p_1^\prime}2\cdot\min(\frac\delta d,\frac {s^\prime}n)}\preceq
c_{k}\Big({\rm id}, \ell_{q_1}(2^{j\delta}\ell_{p_1}(w)),
\ell_{q_2}(\ell_{p_2})\Big)\preceq k^{-\frac {p_1^\prime}2
\cdot\min(\frac\delta d,\frac {\delta+s^\prime}{2n})}.
\end{equation}
 \vspace{-0.4cm}
\item[$(iv)$] If\ $\delta<s^\prime,\ \delta+ s^\prime<\frac n{p_1^\prime}$ and
$\delta-s^\prime>\frac{2d-n}{p_1^\prime}$, then \vspace{-0.2cm}
\begin{equation}k^{-\frac
{p_1^\prime}2\cdot\min(\frac\delta d,\frac {s^\prime}n)}\preceq
c_k\Big({\rm id}, \ell_{q_1}(2^{j\delta}\ell_{p_1}(w)),
\ell_{q_2}(\ell_{p_2})\Big)\preceq k^{-\frac
{p_1^\prime}n\cdot\min(\delta,\frac {\delta+s^\prime}4)}.
\end{equation}
 \end{enumerate}
\end{prop}

\begin{proof}
We consider three cases from the standpoint of $p_1$\ and\ $p_2$.
Note that points (ii)-(iv) vanish when $0<p_1\le 1$.

Case 1.\ Assume $1<  p_1 < 2 < p_2 \le \infty$. This is
corresponding to Case 1 in the proof of Proposition \ref{kn4}. So we
may deal with the proof exactly in the same manner in terms of Lemma
\ref{gn1} (ii). The changes begin with (\ref{andiscr}), where $d_n$
is substituted by $c_n$. And the others go on trivially.

Case 2.\ Assume $0<p_1\le 1$ and $2 <p_2\le\infty$. We proceed as
above. Related computations of ideal quasi-norms herein are finished
by  Lemma \ref{gnupp} (i). Moreover, in the estimate of lower
bounds, because of $0<p_1\le 1$, we consider two cases,
$0<\frac\delta d\le\frac {s^\prime}n$\
 or $0<\frac {s^\prime}n<
\frac\delta d$, and take $m=\left[\frac N2\right]$ in both cases
based on (\ref{gnlow2}).
\end{proof}

\begin{prop}\label{gn5}
Suppose $ 0< p_1 < p_2 \le 2$. We set $\theta_1 =
\frac{1/{p_2^\prime}-1/{p_1^\prime}}{1/2-1/{p_1^\prime}}$.
 \beg{enumerate}
\item[$(i)$] If\, $\min\big(\frac\delta d,\frac
{s^\prime}n\big)>\frac{\theta_1}{p_1^\prime}$, then
 \vspace{-0.2cm}
\begin{equation}
c_{k}\Big({\rm id}, \ell_{q_1}(2^{j\delta}\ell_{p_1}(w)),
\ell_{q_2}(\ell_{p_2})\Big) \sim k^{-\min\big(\frac\delta d,\frac
{s^\prime}n\big)
 -\frac 1{p_1}+\frac 1{p_2}}.
\end{equation}
 \vspace{-0.4cm}
\item[$(ii)$] If $\delta > s^\prime$ and either
 $
\begin{cases}
\delta<\frac d{p_1^\prime}\theta_1,\\
\delta-s^\prime<\frac{2d-n}{p_1^\prime}\theta_1,
\end{cases}\ \ {\rm or}\ \
\begin{cases}
\delta+s^\prime<\frac n{p_1^\prime}\theta_1,\\
\delta-s^\prime>\frac{2d-n}{p_1^\prime}\theta_1,
\end{cases} $
then
\begin{equation}
c_{k}\Big({\rm id}, \ell_{q_1}(2^{j\delta}\ell_{p_1}(w)),
\ell_{q_2}(\ell_{p_2})\Big) \sim k^{-\frac
{s^\prime}n\cdot\frac{p_1^\prime}2}.
\end{equation}
 \vspace{-0.4cm}
\item[$(iii)$] If\ $\delta<\frac d{p_1^\prime}\theta_1,\
\delta<s^\prime<\frac n{p_1^\prime}\theta_1$ and
$\delta-s^\prime<\frac{2d-n}{p_1^\prime}\theta_1$, then
\vspace{-0.2cm}
\begin{equation}k^{-\frac
{p_1^\prime}2\cdot\min(\frac\delta d,\frac {s^\prime}n)}\preceq
c_{k}\Big({\rm id}, \ell_{q_1}(2^{j\delta}\ell_{p_1}(w)),
\ell_{q_2}(\ell_{p_2})\Big)\preceq k^{-\frac {p_1^\prime}2
\cdot\min(\frac\delta d,\frac {\delta+s^\prime}{2n})}.
\end{equation}
 \vspace{-0.4cm}
\item[$(iv)$] If\ $\delta<s^\prime,\ \delta+ s^\prime<\frac n{p_1^\prime}\theta_1$ and
$\delta-s^\prime>\frac{2d-n}{p_1^\prime}\theta_1$, then
\vspace{-0.2cm}
\begin{equation}k^{-\frac
{p_1^\prime}2\cdot\min(\frac\delta d,\frac {s^\prime}n)}\preceq
c_k\Big({\rm id}, \ell_{q_1}(2^{j\delta}\ell_{p_1}(w)),
\ell_{q_2}(\ell_{p_2})\Big)\preceq k^{-\frac
{p_1^\prime}n\cdot\min(\delta,\frac {\delta+s^\prime}4)}.
\end{equation}
 \end{enumerate}
\end{prop}

\begin{proof}
We consider two cases for $p_1$\ and\ $p_2$. Note that points
(ii)-(iv) vanish when $0<p_1\le 1$.

Case 1.\ Assume  $1<  p_1 < p_2 \le 2$. This is corresponding to
Case 1 in the proof of Proposition \ref{kn5}. So we may deal with
the proof exactly in the same manner in terms of Lemma \ref{gn1}
(iv) and the ideas from Case 1 of the last proof.

Case 2.\ Assume $0<p_1\le 1$\ and\ $ p_1< p_2\le 2$. Once more we
proceed exactly as in the proof of Theorem 6 in \cite{LS11}. Here,
Lemma \ref{gnupp} (ii) implies the computations of corresponding
operator ideal quasi-norms. In the lower estimate,
 we consider two cases,
$0<\frac\delta d\le\frac {s^\prime}n$\
 or $0<\frac {s^\prime}n<
\frac\delta d$, as in Case 2 of the last proof, and take
$m=\big[\frac N2\big]$ in both cases based on Lemma \ref{gnlow} (ii)
or (\ref{gnlow0}).
\end{proof}

\begin{prop}\label{gn6}
Suppose $2\le p_1 \le p_2\le \infty$ \,\,or\,\ $0< p_1= p_2< 2$.
Then
 \vspace{-0.2cm}
\begin{equation}
c_{k}\Big({\rm id}, \ell_{q_1}(2^{j\delta}\ell_{p_1}(w)),
\ell_{q_2}(\ell_{p_2})\Big) \sim  k^{-\varkappa},
\end{equation}
with
\begin{equation}
\varkappa=\min\big(\frac\delta d,\frac {s^\prime}n\big).
\end{equation}
\end{prop}

\begin{proof}
The proof of this proposition follows exactly as in the proof of
Proposition \ref{an5} with Lemma \ref{an1} replaced by Lemma
\ref{gn1} and Lemma \ref{gn021ify}.
\end{proof}

\begin{prop}\label{gn7}
Suppose $0<\tilde{p}<p_2<p_1\leq\infty$ where $\frac
1{\tilde{p}}=\min\big(\frac\delta d,\frac {s^\prime}n\big)+\frac
1{p_1}$. Then
 \vspace{-0.2cm}
\begin{equation}
c_{k}\Big({\rm id}, \ell_{q_1}(2^{j\delta}\ell_{p_1}(w)),
\ell_{q_2}(\ell_{p_2})\Big) \sim  k^{-\varkappa},
\end{equation}
with
\begin{equation}
\varkappa=\min\big(\frac\delta d,\frac {s^\prime}n\big)+\frac
1{p_1}-\frac 1{p_2}.
\end{equation}
\end{prop}

\begin{proof}
The proof of this proposition follows exactly as in the proof of
Proposition \ref{an6} with Lemma \ref{an4} replaced by Lemma
\ref{gn021ify}.
\end{proof}

\begin{re}\label{q1q2}
It is remarkable that all results in Section \ref{snss} are
independent of\ $q_1$ and\ $q_2$.
\end{re}

\section{Widths of embeddings of 2-microlocal Besov
Spaces}\label{wmb}

Using basic properties of these $s$-numbers and Proposition
\ref{iso}, we have, for any $s\in\{a,c,d\}$,
$$
s_k\Big({\rm id}, B_{p_1,q_1}^{s_1,s_1^\prime}(\mathbb{R}^n, U),
B_{p_2,q_2}^{s_2,s_2^\prime}(\mathbb{R}^n, U)\Big)\sim s_k\Big({\rm
id},\ell_{q_1}\big(2^{j(s_1-\frac
n{p_1})}\ell_{p_1}(v_1)\big),\ell_{q_2}\big(2^{j(s_2-\frac
n{p_2})}\ell_{p_2}(v_2)\big)\Big),
$$
 with equivalence constants independent of $k\in\mathbb{N}$, cf. also (\ref{AABB}).
This leads us to transfer the results of Section \ref{snss} for
sequence spaces back to function spaces. Theorem \ref{T1} follows
from Proposition \ref{an4}, Proposition \ref{an5} and Proposition
\ref{an6}. Theorem \ref{T2} follows from Propositions
\ref{kn4}\,-\,\ref{kn7}. Theorem \ref{T3} follows from Propositions
\ref{gn4}\,-\,\ref{gn7}.\qed

\begin{re}
If\ \ $U=\{x_0\}$, similar conclusions on the approximation, Gelfand
and Kolmogorov numbers could be made for Corollary 8 in \cite{LS11}.
\end{re}
\begin{re}
If\ $U=\{0\}$, the comparison between our main theorems and the
known results on the approximation, Gelfand and Kolmogorov numbers
of embeddings of Besov spaces with polynomial weights, cf.
\cite{Sk05,ZF10}, could also be easily made, as is shown for entropy
numbers in Remark 4 of \cite{LS11}. We do not go into detail.
\end{re}

\begin{re}
Finally, we wish to mention some open questions. What is the
asymptotic behavior of related n-widths for the other cases unsolved
here (especially the case $\frac\delta d\le\frac
1{\min(p_1^\prime,p_2)}\le\frac {s^\prime}n$ with
$0<p_1<2<p_2\le\infty$ for the approximation numbers), under the
equivalent condition of compactness, $\min(\frac\delta d,\frac
{s^\prime}n)>\frac 1{p^\ast}$? In some cases, the optimal order may
even depend on the microscopic parameters $q_1$ and $q_2$.
\end{re}

\section*{Acknowledgments}
~~~ The authors wish to thank Thomas K$\ddot{\rm u}$hn for sending
us historical remarks concerning the operator ideal technique, which
led us to add Remark \ref{opt}.

The authors are also grateful to Fanglun Huang, Erich Novak and
Leszek Skrzypczak for their direction and help on this work.

 \hspace{5mm}

\end{document}